\documentclass[leqno,11pt]{amsart}
\usepackage[colorlinks,pagebackref,hypertexnames=false]{hyperref}

\usepackage{amsmath,amsthm,amssymb,mathrsfs} 
\usepackage{tikz-cd}
\usepackage[alphabetic,backrefs]{amsrefs}
\usepackage{ae,aecompl}
\usepackage[margin=1in]{geometry} 
\usepackage{tikz}
\usepackage[all]{xy}
\usepackage{url}
\usepackage{comment}

\usepackage[english]{babel}

\usepackage{bookmark}

\numberwithin{equation}{section}

\newcommand{\U}{{\rm U}}

\renewcommand{\epsilon}{\varepsilon}

\renewcommand{\Re}{\mathop{\mathrm{Re}}}

\newcommand{\vol}{\mathrm{vol}}

\def\<{\mathopen{}\left<}
\def\>{\right>\mathclose{}}
\def\({\mathopen{}\left(}
\def\){\right)\mathclose{}}

\newtheorem{theorem}{Theorem}

\newtheorem{corollary}[theorem]{Corollary}

\newtheorem{proposition}[theorem]{Proposition}
\newtheorem{lemma}[theorem]{Lemma}

\theoremstyle{definition}

\newtheorem{example}[theorem]{Example}

\newtheorem{remark}[theorem]{Remark}

\theoremstyle{definition}

\newtheorem{definition}[theorem]{Definition}

\numberwithin{theorem}{section}
 
\numberwithin{equation}{section}
\numberwithin{figure}{section}

\newcommand{\Comment}[2][\empty]{\ifthenelse{\equal{#1}{\empty}}{\todo[color=gray!10]{#2}}{\todo[color=gray!10,#1]{#2}}}

\author{Jason D. Lotay} 
\address[Jason D. Lotay]{University of Oxford, U.K.}
\urladdr{\href{http://people.maths.ox.ac.uk/lotay/}{http://people.maths.ox.ac.uk/lotay/}}
\email{jason.lotay@maths.ox.ac.uk}

\author{Goncalo Oliveira} 
\address[Gon\c{c}alo Oliveira]{Department of Mathematics and CAMGSD, Instituto Superior T\'ecnico, Portugal}
\urladdr{\href{https://sites.google.com/view/goncalo-oliveira-math-webpage/home}{https://sites.google.com/view/goncalo-oliveira-math-webpage/home}}
\email{{goncalo.m.f.oliveira@tecnico.ulisboa.pt}}
\email{{galato97@gmail.com}}

\title[LMCF and circle symmetry]{Neck pinch singularities  and Joyce conjectures 
in Lagrangian mean curvature flow with circle symmetry} 

\date{}

\begin{document}

	\begin{abstract}
		In this article we consider the Lagrangian mean curvature flow of compact, circle-invariant, almost calibrated Lagrangian surfaces in hyperk\"ahler 4-manifolds with circle symmetry. We show that this Lagrangian mean curvature flow can be continued for all time, through a finite number of finite time singularities, and converges to a chain of special Lagrangians, verifying various aspects of Joyce's conjectures \cite{JoyceConjectures} in this setting. This result provides the first non-trivial setting where Lagrangian mean curvature flow may be used successfully to achieve the desired decomposition of a Lagrangian into a sum of special Lagrangians representing its Hamiltonian isotopy class.  We also show that the singularities of the flow are neck pinches in the   sense conjectured by Joyce \cite{JoyceConjectures} and give   examples where such finite time singularities are guaranteed to occur. 
	\end{abstract}
	
	\maketitle
	\tableofcontents

	\section{Introduction}

	\subsection{Context}

	A standing conjecture of Thomas \cite{Thomas}, motivated by mirror symmetry, asserts that there is a stability condition for compact graded Lagrangians in a Calabi--Yau manifold, which is expected to determine the existence (and uniqueness) of a special Lagrangian in a given Hamiltonian isotopy class. However, this stability condition is  hard to work with and the conjecture has so far remained unproven in its full generality. In \cite{Lotay2020}, the authors proved the Thomas conjecture for circle-invariant Lagrangians in a large class of hyperk\"ahler 4-manifolds which includes all complete examples with finite topology obtained from the Gibbons--Hawking ansatz. This contains all ALE and ALF hyperk\"ahler 4-manifolds admitting a tri-Hamiltonian circle action.
	
	Soon after stating the above conjecture, Thomas--Yau \cite{ThomasYau} proposed that a similar stability condition  on compact graded Lagrangians controls the long-time existence and convergence of the Lagrangian mean curvature flow. Later developments of Neves \cite{NevesSingularities} showed explicitly that the initial Lagrangian must at least be almost calibrated, which had implicitly been assumed in \cite{ThomasYau}, otherwise finite time singularities of the flow are inevitable. Under the almost calibrated assumption the authors proved, also in \cite{Lotay2020}, the circle-invariant Thomas--Yau conjecture in the same class of examples mentioned before. 
	
	Both the Thomas and Thomas--Yau conjectures pre-date Bridgeland's definition of stability conditions on triangulated categories \cites{Bridgeland2007,Douglas2001}, which can be applied to the study of Lagrangians in Calabi--Yau manifolds by using Fukaya categories where special Lagrangians become (semi-)stable objects. Using this perspective, Joyce \cite{JoyceConjectures} updated the Thomas/Thomas--Yau conjectures by making use of the notion of Bridgeland stability condition, now seeking, in good cases, to find a sum of special Lagrangians representing a given Hamiltonian isotopy class.  As a way to tackle this conjecture, Joyce proposed the use of Lagrangian mean curvature flow through finite time singularities (possibly with surgeries).\footnote{See also \cites{Solomon2014,Solomon2017,Solomon2020} for an alternative proposal on how to approach the Thomas and Thomas--Yau conjectures and also \cite{Li2022} for some very recent developments.} In particular, Joyce's program relies on a number of conjectures and principles, most of which are still open, regarding what happens along the Lagrangian mean curvature flow to an unstable Lagrangian. In this article, we shall prove versions of Joyce's conjectures for circle-invariant Lagrangians in the class of hyperk\"ahler 4-manifolds mentioned above, thus providing substantial progress towards the full verification of Joyce's program in this setting. Our results give concrete examples of so-called ``neck pinch'' singularities and establish exactly what happens for the flow of the Lagrangians mentioned above: see Theorem \ref{thm:Lawlor_Intro}. Our main result, Theorem \ref{thm:Flow_Intro}, shows that the flow through neck pinch singularities exists for all time and converges at infinity to a chain of special Lagrangian submanifolds (with possibly different phases).  This provides the first non-trivial example where Lagrangian mean curvature flow can be used successfully to decompose a given Lagrangian into the desired connect sum of special Lagrangians which represents the original Lagrangian's Hamiltonian isotopy class.

	\subsection{Main results}
	
	In this article, we shall prove some of Joyce's key conjectures on the behaviour of the Lagrangian mean curvature flow starting at a (possibly unstable) circle-invariant Lagrangian.
	
	Our first main result explores the local behaviour at the finite time singularities of Lagrangian mean curvature flow. The type of singularity we find is known as a ``neck pinch''  as the resulting flow of Lagrangians locally resembles a family of cylinders $S^1 \times \mathbb{R}$ with $S^1 \times \lbrace 0 \rbrace$ contracted to a point when the singularity forms. This proves versions of the conjectures stated as Principle 3.9(a), Problem 3.12(a), and Conjecture 3.16(ii)-(iii) in \cite{JoyceConjectures}. 
	
	To state the result, we say that a chart $\varphi:U \subseteq X \to V \subseteq \mathbb{C}^2$ on a hyperk\"ahler 4-manifold $X$, viewed as a Calabi--Yau 2-fold, is a \emph{pointed isomorphism at $p\in U$} if $\varphi(p)=0 \in V$ and it identifies the Calabi--Yau structure in $T_pX$ with the standard one in $\mathbb{C}^2$. We also recall that \emph{Lawlor necks} in $\mathbb{C}^2$ are exact, embedded special Lagrangians which are topologically $S^{1}\times\mathbb{R}$ and asymptotic to a pair of special Lagrangian planes which intersect transversely at the origin.
	
	\begin{theorem}[``Neck pinch'' finite time singularities]\label{thm:Lawlor_Intro}
		Let $X$ be an ALE or ALF hyperk\"ahler $4$-manifold admitting a  tri-Hamiltonian circle action.   
		Let $L$ be an embedded, almost calibrated, circle-invariant Lagrangian in $X$ which is either compact or asymptotic at infinity to a pair of planes.
		
		Let $\lbrace L_t \rbrace_{t \in [0,T)}$ be the unique  smooth solution to Lagrangian mean curvature flow in $X$ starting at $L$, so that $L_t$ has the same properties as $L$ for each $t\in [0,T)$, for some $T>0$.
		
		\begin{enumerate}
		\item[(a)]  Suppose that the flow $L_t$  develops a finite time singularity at $p \in X$ when $t \to T<\infty$. 
Then, there exist
\begin{itemize}
\item open neighbourhoods $U$ of $p$ in $X$ and $V$ of $0$ in $T_pX\cong\mathbb{C}^2$;
\item a pointed isomorphism $\varphi:U\to V$ at $p$;
\item a small $\delta>0$ and a smooth function $\epsilon:(T-\delta^2,T)\to (0,\delta)$, with $\epsilon(t)\searrow 0$ as $t\nearrow T$,	
\end{itemize}
such that $\epsilon(t)^{-1}\varphi(L_t\cap U)$ converges on compact subsets of $\mathbb{C}^2$ to a unique Lawlor neck.
		\item[(b)]
		For any such $X\neq \mathbb{R}^3 \times \mathbb{S}^1$ there is such a Lagrangian  $L$ so that the flow $L_t$ starting at $L $ develops a finite type singularity as in \emph{(a)}.  Moreover, if $X$ contains a pair of special Lagrangian spheres with different phases, even up to changing orientations on the spheres, then $L$ can be chosen to be compact.
		\end{enumerate}
	\end{theorem}
	
	\begin{remark}
		As stated in \cite{JoyceConjectures}*{Remark 3.10}, Theorem \ref{thm:Lawlor_Intro}(a) contains more information than stating that the (unique) Type II blow-up at the singularity is a Lawlor neck. It can be interpreted as saying that there is a fixed Lawlor neck $L_\infty$ in  $\mathbb{C}^2$ such that $\varphi^{-1}(\epsilon(t) L_\infty \cap V)$ models the movement of $L_{t}$ in the fixed, definite size neighbourhood $U$ of $p\in X$.
	\end{remark}
	
	Before stating our second main result we require the following definition.
	
	\begin{definition}\label{dfn:Akchain}
		An ordered set $\{L_1 , \ldots , L_k\}$ of $k$ 2-spheres in $X$ is said to be an \emph{$A_k$ chain} if for all $i \neq j$ their intersections satisfy $L_i \cdot L_j = \delta_{i-1,j} + \delta_{i,j-1}$.
	\end{definition}
	
	Our second main result shows that, in our setting, there are a finite number of singular times along the flow and, moreover, a Lagrangian mean curvature flow \emph{through singularities} exists for all time and converges to a finite union of special Lagrangian spheres (with possibly different phases). This proves, in our setting, several parts of the program   in   \cite{JoyceConjectures}*{Section 3.2 \& Conjecture 3.9}.  
	
	\begin{theorem}[Decomposition into special Lagrangians]\label{thm:Flow_Intro}
		Let $X$ be an ALE or ALF hyperk\"ahler $4$-manifold admitting a  tri-Hamiltonian circle action and $L$ a compact,  connected, embedded, circle-invariant, almost calibrated Lagrangian in $X$. There is a continuous family $\lbrace L_t \rbrace_{t \in [0, + \infty)}$ of almost calibrated, circle-invariant, Lagrangian integral currents with $L_0=L$ so that the following holds.
		
		\begin{itemize}
			\item[(a)] There is a finite number of singular times $0<T_1 \leq \ldots \leq T_l<\infty$ such that the family $\lbrace L_t \rbrace_{t \in [0, \infty) \backslash \lbrace T_1, \ldots, T_l \rbrace}$ satisfies Lagrangian mean curvature flow.
			\item[(b)] At each singular time $T_i$ the flow undergoes a ``neck pinch'' singularity as in Theorem \ref{thm:Lawlor_Intro}.		
			\item[(c)] There is $k \in \mathbb{N}$ and an $A_k$-chain of embedded, special Lagrangians spheres $\{L^\infty_1, \ldots , L_k^{\infty}\}$ such that $L_t$  converges uniformly to $\cup_{j=1}^kL_j^{\infty}$ as $t \to + \infty$ and we have current convergence:
			$$\lim_{t \to + \infty} L_t = L_1^{\infty} + \ldots + L_k^{\infty}.$$
			
			\item[(d)] If the grading on $L$ is a perfect Morse function, then the phases $\theta_j$ of the special Lagrangians $L_j^{\infty}$ from \emph{(c)} can be chosen to satisfy $\theta_1\geq \ldots\geq \theta_k$.
		\end{itemize} 
		Furthermore, in \emph{(c)} we have that $k=1$ if $L$ is flow stable, and $k>1$ if $L$ is unstable.
	\end{theorem}
	
	\begin{remark}
		The uniform convergence in (c) can always be improved to smooth convergence unless in the $A_k$-chain of spheres $\{L_1^{\infty},\ldots,L_k^{\infty}\}$ we have that the phase of two adjacent spheres are equal.  In this case there is the possibility of an infinite time singularity of the flow: this is the so-called ``semi-stable'' case.  Therefore, if we assume that we are in the setting where no three singularities of the potential $\phi$ defining the hyperk\"ahler triple on $X$ lie on a line, then we will always have smooth convergence.  We can always make a small perturbation of the hyperk\"ahler structure on $X$ so that we are in this regime.
	\end{remark}
	
To prove our main results we first use the key observation that Lagrangian mean curvature flow in our setting is equivalent to a modified curve shortening flow in the plane \cite{Lotay2020}.  An important challenge is that this flow of planar curves  degenerates precisely at the points where singularities may occur.  

A delicate and detailed blow-up analysis is therefore required to understand the behaviour near finite-time singularities.  In particular, this includes proving improved control of the blow-up rate of the curvature at such a singularity, which shows that the speed of the flow tends to zero as it approaches a singularity, and  that the curve is sufficiently regular at a singular time that the flow may be restarted.  The latter result involves using recent work in \cite{FSS.neck}.

Another important tool is the construction of suitable barrier curves to guarantee finite-time singularities and to control the behaviour of the flow.  These barrier curves are constructed by showing that convexity is preserved along the flow and by exploiting a link between appropriate disks in the plane bounded by the barrier curves (which we call pacman disks) and holomorphic curves with boundaries on evolving Lagrangians in the hyperk\"ahler 4-manifold.  The latter observation crucially allows us to get a uniform bound for the rate of decrease of the area of pacman disks.  

Our methods enable us to prove the first convergence results for Lagrangian mean curvature flow starting at unstable Lagrangians, which requires the first long-time existence results for such a flow through singularities.
	
	\subsection*{Acknowledgements}  The first author would like to thank Dominic Joyce for useful conversations.  The first author was partially supported by EPSRC grant EP/T012749/1 during the course of this project. The second author wants to thank the members of Hausel group at IST Austria who hosted him as part of this project was being performed. The second author is currently funded by FCT 2021.02151.CEECIND and previously by the NOMIS foundation.	
	Both authors would like to thank the Simons Laufer Mathematical Sciences Institute (formerly known as MSRI), Berkeley for hospitality during the latter stages of this project.

	\section{The Gibbons--Hawking ansatz}\label{sec:GH_Ansatz}
	
	We provide a short description of the Gibbons--Hawking ansatz, and circle-invariant Lagrangians and notions of stability in this context, as necessary for our study. For further details, we refer the reader to  \cite{Lotay2020}.

	\subsection{Hyperk\"ahler 4-manifolds with tri-Hamiltonian circle action}
	
	Let $(X^4, g)$ be hyperk\"ahler and equipped with a circle action preserving the three K\"ahler forms $\omega_1,\omega_2,\omega_3$, associated with the orthogonal complex structures $I_1,I_2,I_3$ on $X$ satisfying the quaternionic relations. Denote by $\xi$ the infinitesimal generator of the circle action and let $\hat{X}$ 
	 be the open dense set where the action is free. Then, $\hat{X}$ is a $\U(1)$-bundle
	$$\pi: \hat{X} \rightarrow Y^3 $$
	over an open $3$-manifold $Y^{3}$. Equip this bundle with a connection $\eta \in \Omega^1(\hat{X}, \mathbb{R})$ whose horizontal spaces $\ker(\eta)=\xi^{\perp}$ are $g$-orthogonal to $\xi$ and so that $\eta(\xi)=1$. Note that   $\iota_{\xi} d \eta =0$.
	
	Consider the positive $\U(1)$-invariant function $\phi:\hat{X}\to \mathbb{R}$ determined by $\phi^{-1}=|\xi|_g^{2}$ and define the 1-forms $\alpha_i:= I_i (\phi^{-1} \eta)$, for $i=1,2,3$. The hyperk\"ahler metric $g$ may then be written on $\hat{X}$ as
	\begin{equation}\label{eq:HypMetric}
		g= \phi^{-1} \eta^2 + \alpha_1^2 + \alpha_2^2 + \alpha_3^2 
	\end{equation}
	and the associated K\"ahler forms are given by:
	\begin{align}\label{eq:hypforms}
		\omega_1 & =  \phi^{-1/2} \eta \wedge \alpha_1 + \alpha_2 \wedge \alpha_3 ,\quad 
		\omega_2 
		=  \phi^{-1/2} \eta \wedge \alpha_2 + \alpha_3 \wedge \alpha_1 ,\quad  
		\omega_3 
		=  \phi^{-1/2} \eta \wedge \alpha_3 + \alpha_1 \wedge \alpha_2.
	\end{align} 
	With the orientation induced by the volume form $\phi^{-1} \eta \wedge \alpha_{1}\wedge\alpha_{2}\wedge\alpha_{3}$, the forms $\omega_1,\omega_2,\omega_3$ give a trivialization of $\Lambda^2_+\hat{X}$, the bundle of self-dual 2-forms on $\hat{X}$. Conversely, if we define 2-forms $\omega_1,\omega_2,\omega_3$ as in \eqref{eq:hypforms} and fix the volume form $\phi^{-1} \eta\wedge\alpha_1\wedge\alpha_2\wedge\alpha_3$, we can recover the metric $g$ as in \eqref{eq:HypMetric}, and it follows from \cite{Atiyah}*{Lemma 4.1} that $g$ is hyperk\"ahler if and only if $d\omega_i=0$ for $i=1,2,3$. Using this characterization we have the following (cf.~\cite{Lotay2020}*{Proposition 2.1}).
	
	\begin{proposition}\label{prop:MainHK} 
		Using the notation above, the metric $g$ in \eqref{eq:HypMetric} equips $X^4$ with a hyperk\"ahler structure so that the $\U(1)$ action generated by $\xi$ preserves $g$ and 
		$\omega_1,\omega_2,\omega_3$ in \eqref{eq:hypforms} if and only if the following hold.
		\begin{enumerate}
			\item[(a)] The symmetric $2$-tensor 
			$$g_E=\phi^{-2} (\alpha_1^2+\alpha_2^2+\alpha_3^2)
			$$
			is the pullback of a flat metric on $Y^3$.
			\item[(b)] The pair $(\eta,\phi)$ is a Dirac monopole on $Y^3$, i.e.
			\begin{equation}\label{eq:Bogomolnyi}
				\ast_E d\eta=- d\phi,
			\end{equation}
			where $\ast_E$ denotes the Hodge star operator associated with the metric $g_E$ on $Y^3$ from \emph{(a)}.
			
			\item[(c)] There are local coordinates $(\mu_1,\mu_2,\mu_3)$ on $Y^3$ such that $\alpha_i=\phi^{\frac{1}{2}}d\mu_i$, and the hyperk\"ahler metric can be written as
			\begin{equation}\label{eq:HypMetric2}
				g=\frac{1}{\phi}\eta^2 
				+\phi\left( d \mu_1^2  + d \mu_2^2 + d \mu_3^2 \right).
			\end{equation}
			Moreover,
			\begin{equation}\label{eq:hypforms2}
			\omega_1   =  \eta \wedge d\mu_1 + \phi d\mu_2\wedge d\mu_3 ,\quad 
		\omega_2 
		=  \eta \wedge d\mu_2 + \phi d\mu_3\wedge d\mu_1 ,\quad  
		\omega_3 
		=  \eta \wedge d\mu_3 + \phi d\mu_1\wedge d\mu_2.
			\end{equation}
		\end{enumerate}
	\end{proposition}

	We shall consider the case when $Y^3$ is simply connected, in which case the coordinates $(\mu_1,\mu_2,\mu_3)$ can be taken to be global and form the hyperk\"ahler moment map
	$$\pi: X \rightarrow \mathbb{R}^3.$$
	In \cite{Lotay2020}*{Section 2.2} we give examples of hyperk\"ahler manifolds arising from this construction including the flat, Taub--NUT, Eguchi--Hanson, Ooguri--Vafa and Anderson--Kronheimer--LeBrun examples. Here we shall simply recall the multi-Taub--NUT and multi-Eguchi--Hanson examples.
	
	\begin{example}\label{ex:ALEALF}
		Choose $k\geq 1$ points $p_1, ..., p_k$ in $\mathbb{R}^3$ and $m \geq 0$. Set $Y= \mathbb{R}^3 \backslash \lbrace p_1, ... , p_k \rbrace$ and 
		\begin{equation}\label{eq:phi.MultiEH}
			\phi= m+ \sum_{i=1}^k\frac{1}{2 \vert x - p_i \vert},
		\end{equation}
		where the norm is the Euclidean metric on $\mathbb{R}^3$. 
		Then there is a connection $\eta$ satisfying \eqref{eq:Bogomolnyi} and the metric $g$ given by \eqref{eq:HypMetric2}   extends smoothly across the points $p_1, \ldots , p_k$ where the circle action collapses, thus completing $\hat{X}=\pi^{-1}(\mathbb{R}^3 \backslash \lbrace p_1, \ldots , p_k \rbrace)$ to a hyperk\"ahler manifold $(X^4,g)$ by adding in $k$ points.
		
		Suppose that $m=0$. In this case, when $k=1$ we obtain the flat metric on $\mathbb{R}^4$ and for  $k =2$ we get the Eguchi--Hanson metric. For  $k >2$ the resulting metric is called the multi-Eguchi--Hanson metric.  We always have $\lim_{r\to\infty}\phi=0$ (where $r$ is the distance to a fixed point, say the origin, in $\mathbb{R}^3$) and thus $g$ is asymptotic to the flat metric on $\mathbb{R}^4/\mathbb{Z}_k$, so $(X^4,g)$ is \emph{asymptotically locally Euclidean} (ALE).
		
		Suppose now that $m>0$.  When $k=1$ we obtain the Taub--NUT metric on $\mathbb{R}^4$ and so if $k>1$ we call the resulting metric  multi-Taub--NUT.  (Note that if we allowed $k=0$ we would obtain the product metric on $\mathbb{S}^1\times\mathbb{R}^3$.) In this situation $\lim_{r\to\infty}\phi=m>0$ and so the circle generated by $\xi$ has finite length at infinity. The metric $g$ is therefore asymptotic to one on a circle bundle over $\mathbb{R}^3$ with the fibers having constant length proportional to $m^{-1/2}$. Such metrics are called \emph{asymptotically locally flat} (ALF).
	\end{example}
	
	\begin{remark}
	 Example~\ref{ex:ALEALF} describes all of the ALE and ALF gravitational instantons which arise from the Gibbons--Hawking ansatz.
	\end{remark}

		\subsection{Notation}\label{ss:notation}
	
	Throughout this article we will be working on a hyperk\"ahler 4-manifold $X^4$ given by the Gibbons--Hawking ansatz which is an ALE or ALF gravitational instanton.  We will therefore use the notation in   Proposition \ref{prop:MainHK} and Example~\ref{ex:ALEALF} for the remainder of the article.

	\subsection{Lagrangian submanifolds}
	On any hyperk\"ahler 4-manifold, the twistor sphere defines a 2-sphere of K\"ahler structures for which one can study Lagrangians.  It is important to understand this 2-sphere explicitly in our setting so that we may identify circle-invariant Lagrangians in $X$ with certain curves in the base of the fibration $\pi:X\to \mathbb{R}^3$.  Throughout, we shall only consider connected Lagrangians.
	
	To this end, we see that the twistor sphere of $X$ can be identified with the unit $2$-sphere $\mathbb{S}^2\subseteq\mathbb{R}^3$ as follows. Given $v=(v_1 , v_2 , v_3)\in\mathbb{S}^2 \subseteq\mathbb{R}^3$, we have a complex structure $I_{v}$ on $X$ such that the associated $2$-form using the hyperk\"ahler metric is (recalling \eqref{eq:hypforms2})
	\begin{equation}\label{eq:omega.v}
		\omega_{v}=  \sum_{i=1}^3 v_i \left( \eta \wedge  d\mu_i+  \phi \ d\mu_j \wedge d\mu_k \right) ,
	\end{equation}
	with $(i,j,k)$ denoting a cyclic permutation of $(1,2,3)$.  For example, if $v=(0,0,1)$ then $\omega_v=\omega_3$. 
	
	Any circle-invariant surface $L$ in $X$ corresponds, via $\pi:X\to\mathbb{R}^3$, to a curve $\gamma \subseteq \mathbb{R}^3$. 
	For $v \in \mathbb{S}^2$, a short computation (cf.~\cite{Lotay2020}*{Section 5.1}) yields, for $L=\pi^{-1}(\gamma)$,
	$$\omega_{v} |_{L} = \langle \gamma' , v \rangle \vol_{L},$$
	where $\gamma'$ is the velocity of $\gamma$ with respect to Euclidean arclength and $\langle.,.\rangle$ is the Euclidean inner product. We deduce the following.
	
	\begin{lemma} In the notation above, a circle-invariant surface $L^2=\pi^{-1}(\gamma)$ in $X^4$ for a curve $\gamma\subseteq\mathbb{R}^3$ is Lagrangian with respect to $\omega_v$ in \eqref{eq:omega.v} if and only if $\gamma$ lies in a plane orthogonal to $v$.
	\end{lemma}
	
Let $S_\phi=\{p_1,\ldots,p_k\}\subseteq\mathbb{R}^3$ be the set of singularities of $\phi$ as in Example~\ref{ex:ALEALF}.  We then see that for $L=\pi^{-1}(\gamma)$ to be compact  and embedded $\gamma$ cannot have self intersections and is either:
	\begin{itemize}
		\item a simple closed curve not intersecting $S_\phi$, in which case $L=\pi^{-1}(\gamma)\cong T^2$; or
		\item a simple arc with end points $p_i,p_j\in S_{\phi}$ with $i\neq j$ and otherwise not meeting $S_{\phi}$, in which case $\pi^{-1}(\gamma)\cong S^2$.
	\end{itemize}

	\begin{remark}
	 In our setting $S_{\phi}$ is finite, but more generally it can be infinite as in the Ooguri--Vafa and Anderson--Kronheimer--LeBrun metrics (see \cite{Lotay2020}*{Examples 2.7 \& 2.8}).
	\end{remark}

	We will only be interested in graded Lagrangians, as we shall now define. 
	
	\begin{definition}\label{dfn:Lag.angle}
		Given a Calabi--Yau structure $(\omega, \Omega)$ on $X$, determined by a K\"ahler form $\omega$ and holomorphic volume form $\Omega$, an oriented Lagrangian $L$ in $X$ is said to be \emph{graded} by $\theta:L \to \mathbb{R}$ if the restriction of $\Omega$ to $L$ satisfies 
\begin{equation}\label{eq:Lag.angle}
e^{-i\theta} \Omega|_L = \vol_L,
\end{equation}
		where $L$ is the Riemannian volume form associated with the induced metric. We will denote a Lagrangian $L$ graded by $\theta$ as $(L, \theta)$ where necessary and we will refer to $\theta$ as the grading.  The choice of grading $\theta$ is also called the \emph{Lagrangian angle} of $L$. 
	\end{definition}
	
	\begin{remark}\label{rmk:arg.zero}
		Notice, in particular, that for a compact almost calibrated Lagrangian $L$ the quantity
		$$\arg \int_L \Omega$$
		is well-defined up to integer multiples of $2\pi$.  Moreover, given such an $L$ we can always multiply $\Omega$ by a unit complex number so that $\arg\int_L\Omega=0$ modulo $2\pi$.
	\end{remark}

	It will be useful to study a distinguished subclass of the graded Lagrangians as follows.
	
	\begin{definition} Let $(\omega,\Omega)$ be a Calabi--Yau structure on $X$.
An oriented Lagrangian $L$ in $(X,\omega,\Omega)$ is \emph{almost calibrated} if there is a choice of grading $\theta$ on $L$ such that, for some $\delta>0$, 
$$\sup_L\theta-\inf_L\theta\leq \pi-\delta.$$
If $L$ is compact, $L$ is almost calibrated if and only if there is a constant $\theta_0$ so that $\Re(e^{-i\theta_0}\Omega)|_L>0$.
	\end{definition}

	In our situation, given $ v\in\mathbb{S}^2$ and $\omega_v$ as in \eqref{eq:omega.v}, there is a circle of holomorphic volume forms $\Omega_v$ so that $(X,\omega_v,\Omega_v)$ is Calabi--Yau, namely 
\begin{equation}\label{eq:Omega.v}
\Omega_v=\omega_{v_1}+i\omega_{v_2},
\end{equation}
	where $\{v,v_1,v_2\}$ is a positively oriented orthonormal basis for $\mathbb{R}^3$.  For example, if $v=(0,0,1)$ then $\Omega_v=e^{i\alpha}(\omega_1+i\omega_2)$ for some $e^{i\alpha}\in\mathbb{S}^1$. We see that
	$$\Omega_v|_L= \left( \langle \gamma' , v_1 \rangle + i \langle \gamma' , v_2 \rangle \right)  \vol_{L} $$
	and so the Lagrangian angle of $L$ satisfies $\cos \theta =\langle \gamma' , v_1 \rangle $ and $\sin \theta = \langle \gamma' , v_2 \rangle$, i.e.~$\theta$ coincides (mod $2 \pi$) with the angle that $\gamma'$ makes with $v_1$. See   \cite{Lotay2020}*{Section 5} for more details.   As a consequence, any compact, embedded, graded Lagrangian of the form $L=\pi^{-1}(\gamma)$ must have $\gamma$ be a simple arc joining two singularities of $\phi$ and meeting no other singularities, in which case $L$ is   a $2$-sphere.  Moreover, such an $L$ is almost calibrated if and only if the variation of the angle which $\gamma'$ makes with $v_1$ is strictly less than $\pi$: we shall call such curves almost calibrated, by abuse of notation.  Note that almost calibrated curves are automatically embedded.

\begin{remark}\label{rmk:CY}  Suppose that 
$L=\pi^{-1}(\gamma)$ is an embedded, compact, almost calibrated, circle-invariant Lagrangian in $X$, for some choice of Calabi--Yau structure $(\omega,\Omega)=(\omega_v,\Omega_v)$ for some $v\in\mathbb{S}^2$ as above. With no loss of generality we can suppose that $\gamma$ is perpendicular to the $\mu_3$-axis and $\Omega=\omega_1 + i \omega_2$. Furthermore, up to a translation we may set the initial point of $\gamma$ to be $(0,0,0)$ and denote its final point by $(x,y,0)$. Then,
	$$\int_L \Omega = 2 \pi \int_{\gamma} ( d\mu_1 + i d \mu_2 ) = 2 \pi (x + i y),$$
	and so $\arg \int_L \Omega$ is the angle between the straight-line $\overline{\gamma}$ connecting the endpoints of $\gamma$ and the $\mu_1$-axis.    Moreover, 
	$$ \left| \int_L \Omega \right| = 2\pi \sqrt{x^2+y^2}=2\pi \mathrm{Length(\overline{\gamma})},$$
 	and $2\pi \mathrm{Length(\overline{\gamma})}= \mathrm{Area}(\overline{L})$ where $\overline{L}=\pi^{-1}(\overline{\gamma})$ is the area-minimizer in the homology class of $L$. We also notice that $\arg \int_L \Omega =0$ if and only if $y=0$ and $x>0$, i.e.~the endpoint of $\gamma$ lies on the positive $\mu_1$-axis.
\end{remark}

\subsection{Further notation} We shall from now on fix a circle-invariant Calabi--Yau structure $(\omega,\Omega)=(\omega_v,\Omega_v)$ on $X$ for some $v\in\mathbb{S}^2$ as in \eqref{eq:omega.v} and \eqref{eq:Omega.v}. We see from our discussion above that for there to be any compact, embedded, graded, circle-invariant Lagrangians in $(X,\omega,\Omega)$ we require there to be at least 2 singularities of $\phi$, i.e.~we need $k>1$ in the notation of Example~\ref{ex:ALEALF}. 
	
	\subsection{Stability and flow stability}
	
	We shall be using the notions of stability for Lagrangians introduced by Thomas \cite{Thomas} and Thomas--Yau \cite{ThomasYau}. After the introduction of the notion of Bridgeland stability conditions \cite{Bridgeland2007} such notions may be modified in order to use such a framework, but we shall not pursue this here. A key reason for this is that it is not yet known whether Bridgeland stability conditions exist in some version of the Fukaya category relevant for our study.  Moreover, it is part of Joyce's programme \cite{JoyceConjectures} that one should use Lagrangian mean curvature flow to even define the Bridgeland stability condition.  For further discussion of the possible relation between Bridgeland stability conditions and (modifications of) the Thomas--Yau conjecture,  we refer the reader to \cites{JoyceConjectures, Li2022}.
	
	\begin{definition}[Stability]\label{def:Sability}\label{dfn:stability}
		Let  $(L, \theta)$ be a compact graded Lagrangian in  
		$(X,\omega,\Omega)$. Then, $L$ is \emph{unstable} if its Hamiltonian isotopy class can be decomposed as a graded connect sum $L_1 \# L_2$, where $L_1,L_2$ are compact graded Lagrangians with variations of their gradings less than $2\pi$ and  
		$$\arg \int_{L_1} \Omega \geq \arg \int_{L_2} \Omega . $$
		Moreover, if strict inequality can be achieved $(L,\theta)$ is said to be \emph{strictly unstable}. If only equality occurs, then $(L,\theta)$ is called \emph{semi-stable}.
		
		Finally, the compact graded Lagrangian $(L, \theta)$ is called \emph{stable} if it is not unstable.
	\end{definition} 

	\begin{definition}[Flow stability]\label{dfn:flow_stability}  
		Let $(L, \theta)$ be a compact, almost calibrated Lagrangian in  
		$(X,\omega,\Omega)$ satisfying $\arg \int_L \Omega=0$.\footnote{Recall by Remark \ref{rmk:arg.zero} that there is essentially no loss of generality here.}  
		 Then $(L,\theta)$ is \emph{flow stable} if for any possible decomposition of the Hamiltonian isotopy class of $(L, \theta)$ as a graded Lagrangian connect sum $L_1 \# L_2$ for $L_1$ and $L_2$ compact  and almost calibrated, we have
		\begin{itemize}
			\item[(a)] $\left[ \arg \int_{L_1} \Omega ,  \arg \int_{L_2} \Omega \right]\nsubseteq (\inf_L\theta,\sup_L\theta)$, or
			\item[(b)] $\mathrm{Area}(L) < | \int_{L_1} \Omega | + |  \int_{L_2}\Omega |$ .
		\end{itemize}
		Note that the condition (b) implies that $\mathrm{Area}(L)< \mathrm{Area}(L_1)+\mathrm{Area}(L_2)$. We say that $(L,\theta)$ is \emph{flow unstable} if it is not flow stable (i.e.~(a) and (b) are both violated), and \emph{strictly flow unstable} if $\arg\int_{L_1}\Omega<\arg\int_{L_2}\Omega$ in (a) or the reverse of the inequality in (b) holds strictly.  If $(L,\theta)$ is not strictly flow unstable we say it is \emph{flow semi-stable}.
	\end{definition}  

 Recall the notation from Remark \ref{rmk:CY}, where $L=\pi^{-1}(\gamma)$ is a compact, almost calibrated, embedded, circle-invariant Lagrangian in $X$   with projection $\pi:X\to\mathbb{R}^3$.	 We shall now recast the notions of stability/flow stability in Definitions \ref{def:Sability} and \ref{dfn:flow_stability} in this circle-invariant setting, i.e.~in terms of curves.   
 
 Consider any decomposition of the curve $\gamma$ as $\gamma_1 \# \gamma_2$, for $\gamma_1,\gamma_2$ almost calibrated, and denote by $\overline{\gamma_1}, \overline{\gamma_2}$ the straight-lines with the same endpoints and orientations as $\gamma_1$ and $\gamma_2$ respectively. With no loss of generality suppose that the endpoints  of $\gamma$ are on the $\mu_1$-axis, oriented so the $\mu_1$-coordinate increases from one endpoint of $\gamma$ to the other. Let $\theta$, $\overline{\theta}_1$, $\overline{\theta}_2$ denote the angles that $\gamma$, $\overline{\gamma}_1$, $\overline{\gamma}_2$ make with the $\mu_1$-axis respectively. Then we have the following observation.
 
 \begin{lemma}\label{lem:curve.stability} The graded curve $\gamma$, equivalently $L=\mu^{-1}(\gamma)$, is stable if  for all decompositions $\gamma=\gamma_1\#\gamma_2$ into graded curves we have (in the notation above)
	$$\overline{\theta}_1 < \overline{\theta}_2 .$$
	Similarly, an almost calibrated curve $\gamma$ is flow stable if for all decompositions $\gamma=\gamma_1\#\gamma_2$ we have (in the notation above)  
	\begin{itemize}
		\item[(a)] $[ \min \lbrace \overline{\theta}_1,\overline{\theta}_2 \rbrace , \max \lbrace \overline{\theta}_1,\overline{\theta}_2 \rbrace ]\nsubseteq(\inf_{\gamma}\theta,\sup_{\gamma}\theta)$, or
		\item[(b)] $\mathrm{Length}(\gamma) < \mathrm{Length}(\overline{\gamma_1}) + \mathrm{Length}(\overline{\gamma_2})$.
	\end{itemize} 
\end{lemma}

\noindent The notions of flow unstable, strictly flow unstable and flow semi-stable clearly extend from Definition \ref{dfn:flow_stability} to graded curves, following Lemma \ref{lem:curve.stability}.

	\section{Lagrangian mean curvature flow and modified curve shortening flow}

	The goal of this section is to give some preliminary general results regarding the evolution of circle-invariant Lagrangians under mean curvature flow, or equivalently curves under the modified curve shortening flow \eqref{eq:Modified_Curve_Shortening} below. 
	
	Recall, from \cite{Lotay2020}*{Proposition 4.5}, that a circle-invariant Lagrangian $L_t=\pi^{-1}(\gamma_t)$ in $X$    evolves through the Lagrangian mean curvature flow if and only if the planar curve $\gamma_t$ satisfies
	\begin{equation}\label{eq:Modified_Curve_Shortening}
		\partial_t \gamma_t = \phi^{-1} \partial^2_s 
		\gamma_t,
	\end{equation}
	where  $s$ the Euclidean arc-length parameter and $\phi$ is the potential on $X$ as in Proposition~\ref{prop:MainHK}.  We will assume for ease of notation that $\gamma_t \subset \{0\}\times\mathbb{R}^2\subseteq\mathbb{R}^3$ and we identify $\{0\}\times\mathbb{R}^2\cong\mathbb{R}^2\cong\mathbb{C}$.
	
	\subsection{Evolution of the grading and curvature}
	
	Recall  that  for a graded Lagrangian $(L,\theta)$ the Lagrangian angle $\theta$ satisfies \eqref{eq:Lag.angle}.  
	Under Lagrangian mean curvature flow, the grading $\theta_t$ of $L_t$ evolves via 
	\begin{equation}\label{eq:beta.evol}
		\frac{\partial\theta_t}{\partial t}=-\Delta \theta_t=-d^*d\theta_t.
	\end{equation}
	(Here, and throughout, we will use the ``geometer's Laplacian'' $\Delta=d^*d$, so that the evolution equation \eqref{eq:beta.evol} for  $\theta$ is the heat equation.)
	Using \eqref{eq:beta.evol}  we quickly see that
	\begin{align*}
		\left(\frac{\partial}{ \partial t} + \Delta \right)  |d\theta_t|^2  = - 2 |\nabla d \theta_t|^2 .
	\end{align*}
	It turns out that, in our setting, the grading $\theta$ of a circle-invariant Lagrangian $L=\pi^{-1}(\gamma)$ can be related to the curvature $\kappa$ of $\gamma$, defined by $\partial_s^2\gamma=\kappa N$, where  $\lbrace \partial_s\gamma, N \rbrace$ is an oriented orthonormal basis of the plane containing $\gamma$. Viewing $\theta$ as a function on $\gamma$, by \cite{Lotay2020}*{Lemma 5.4} we have  that
\begin{equation}\label{eq:curv.angle}	
	\kappa  = \partial_s \theta .
	\end{equation}
Hence, the evolution equation \eqref{eq:beta.evol} for the grading $\theta_t$ of $L_t = \pi^{-1}(\gamma_t)$ should yield an equation for the evolution of the curvature $\kappa_t$ of $\gamma_t$. This is stated in the next result.
	
	\begin{proposition}\label{prop:Evolution_Curvature}
		Let $\gamma_t$ be a solution of \eqref{eq:Modified_Curve_Shortening} with curvature $\kappa_t$. Then,
		\begin{equation}\label{eq:Evolution_Curvature}
			\partial_t \kappa_t = \partial_s^2(\phi^{-1}\kappa_t) + \phi^{-1}\kappa_t^3 ,
		\end{equation}
		and
		\begin{equation}\label{eq:Evolution_Curvature_Estimate}
			\partial_t (\phi^{-1} \kappa_t) \geq \phi^{-1} \partial_s^2(\phi^{-1}\kappa_t)  + (\kappa_t -2\phi )(\phi^{-1}\kappa_t)^2 . 
		\end{equation}
	\end{proposition}
	\begin{proof}
		Recall that the evolution equation \eqref{eq:Modified_Curve_Shortening} can be written as $$\partial_t \gamma_t = \phi^{-1} \kappa_t N_t$$ where, if $I$ denotes multiplication by $i$ in $\mathbb{C}$ and $'=\partial_s$,  
		$$N_t=I \gamma_t'\quad\text{and}\quad\kappa_t=\langle \gamma_t'' ,N_t \rangle .$$ Now, in order to commute derivatives we use a fixed parameter $x(s)$ of $\gamma_t$ independent of $t$. Furthermore, this may be done so that at a fixed space-time point $(t_0,x_0)$ we have $x'(s)=1$ and $x''(s)=0$, or equivalently $|\partial_x\gamma_t|=1$ and $\langle \partial_x^2 \gamma_t , \partial_x \gamma_t \rangle=0$. Using such a parametrization we have $$\kappa_t= |\partial_x\gamma_t|^{-2} \langle \partial_x^2 \gamma_t , N_t \rangle$$ (noting that $\langle \partial_x\gamma_t,N_t\rangle=0$ at all space-time points) and so
		\begin{align}\label{eq:Derivative_Of_Curvature}
			\partial_t \kappa_t = |\partial_x\gamma_t|^{-2} \langle \partial_t \partial_x^2 \gamma_t , N_t \rangle + |\partial_x\gamma_t|^{-2} \langle \partial_x^2 \gamma_t , \partial_t N_t \rangle - 
			2|\partial_x\gamma_t|^{-2}  \langle \partial_x\gamma_t , \partial_t \partial_{x} \gamma_t \rangle  \kappa_t.
					\end{align}
		Before continuing we make a few elementary observations which will prove useful during the computation. As $\langle \partial_x\gamma_t , N_t \rangle=0$, we find that $$\langle \partial_x^2\gamma_t , N_t \rangle + \langle \partial_x\gamma_t , \partial_x N_t \rangle =0.$$ Since $\langle N_t , \partial_x N _t\rangle =0$, we deduce that
		\begin{align}\label{eq:Derivative_Normal}
			\partial_x N_t = - \kappa_t \partial_x\gamma_t.
		\end{align}
		We now compute each term of \eqref{eq:Derivative_Of_Curvature} separately at the point $(t_0,x_0)$ in question.  For the first term we have, using \eqref{eq:Modified_Curve_Shortening} and \eqref{eq:Derivative_Normal},
		\begin{align*}
			\langle \partial_t \partial_x^2 \gamma_t , N_t \rangle & = \langle \partial_x^2 (\phi^{-1}\kappa_t N_t) , N_t \rangle \\
			& = \partial_x \langle \partial_x (\phi^{-1}\kappa_t N_t) , N_t \rangle - \langle \partial_x (\phi^{-1}\kappa_t N_t) , \partial_x N_t \rangle \\
			& = \partial_x \left( \partial_x \langle \phi^{-1}\kappa_t N_t, N_t \rangle - \langle  \phi^{-1}\kappa_t N_t , \partial_x N_t \rangle  \right)  - \phi^{-1}\kappa_t^3 \\
			& = \partial_x^2(\phi^{-1}\kappa_t) - \phi^{-1}\kappa_t^3 
		\end{align*}
		at $(t_0,x_0)$.  
		Since $\langle N_t,\partial_tN_t\rangle =0$ and $\partial^2_x\gamma_t$ is a multiple of $N_t$ at $(t_0,x_0)$, we see that the second term in \eqref{eq:Derivative_Of_Curvature} vanishes there:
		\begin{align*}
			\langle \partial_x^2 \gamma_t , \partial_t N_t \rangle   =0.
		\end{align*}
		For the last term we again use \eqref{eq:Modified_Curve_Shortening} and \eqref{eq:Derivative_Normal} and find that at $(t_0,x_0)$ we have
		\begin{align*}
			\langle \partial_x\gamma_t , \partial_t \partial_{x} \gamma_t \rangle & = \langle \partial_x\gamma_t , \partial_{x} (\phi^{-1}\kappa_t N_t) \rangle = \langle \partial_x\gamma_t ,  \phi^{-1}\kappa_t \partial_{x} N_t \rangle = - \phi^{-1} \kappa_t^2,
		\end{align*}
		so the last term in \eqref{eq:Derivative_Of_Curvature} is given by $2\phi^{-1}\kappa_t^3$. Inserting all these formulae gives \eqref{eq:Evolution_Curvature}.
			
		For the estimate \eqref{eq:Evolution_Curvature_Estimate} we observe first that, since 
		$$ \phi=m+\sum_{i=1}^k\frac{1}{2r_i}$$
		where $r_i$ is the Euclidean distance to $p_i$, we have  
		$$|\nabla\phi|=|\sum_{i=1}^k\frac{d r_i}{2r_i^2}|\leq \sum_{i=1}^k\frac{1}{2r_i^2}\leq 2\phi^2$$
		because $|d r_i|=1$ for all $i$.  Therefore, using this estimate together with \eqref{eq:Modified_Curve_Shortening} and \eqref{eq:Evolution_Curvature}   we compute
		\begin{align*}
			\partial_t(\phi^{-1}\kappa_t) & = -\phi^{-2} (\partial_t \phi) \kappa_t + \phi^{-1} \partial_t \kappa_t \\
			& =  -\phi^{-2} (\nabla_{\partial_t \gamma} \phi )  \kappa_t + \phi^{-1} ( (\phi^{-1} \kappa_t)'' + \phi^{-1} \kappa_t^3 )  \\
			& =  - \phi^{-3} (\nabla_{N_t} \phi ) \kappa_t^2 + \phi^{-1} (\phi^{-1} \kappa_t)'' + \phi^{-2} \kappa_t^3 \\
			& \geq - 2\phi^{-1} \kappa_t^2 + \phi^{-1} (\phi^{-1} \kappa_t)'' + \phi^{-2} \kappa_t^3.
		\end{align*}
		This gives \eqref{eq:Evolution_Curvature_Estimate}.
	\end{proof}
	
	We want to show the relationship between Proposition \ref{prop:Evolution_Curvature} and convexity of curves along the flow \eqref{eq:Modified_Curve_Shortening}.   To do this, we make a definition.
	
	\begin{definition}
		Let $\gamma\subseteq\mathbb{R}^3$ be an embedded connected planar curve.  We say that $\gamma$ is \emph{convex} if it is a subset of a curve bounding a convex region in the plane containing $\gamma$, which is equivalent to saying that the curvature $\kappa$ of $\gamma$ is either everywhere non-negative or everywhere non-positive.  We say that $\gamma$ is \emph{strictly convex} if $|\kappa|>0$ at every interior point of $\gamma$.  
	\end{definition}
	
	\begin{remark}
		Let $\gamma\subseteq\mathbb{R}^3$ be an embedded planar arc connecting two singularities of $\phi$ and meeting no other singularities of $\phi$.   
		By \eqref{eq:curv.angle}, we see that $\gamma$ is strictly convex if and only if the Lagrangian angle $\theta$ has exactly two critical points  
		on $L=\pi^{-1}(\gamma)$.  
	\end{remark}

	\begin{proposition}\label{prop:convexity}
		Let $\gamma_0 \subseteq \mathbb{R}^3$ be an embedded planar arc connecting two singularities $p_1,p_2$ of $\phi$ and which meets no other singularities of $\phi$.  Suppose further that $\gamma_0$ is strictly convex.   
		
		Let $\lbrace \gamma_t \rbrace_{t \in [0,T]}$, for $T>0$, be a smooth solution of \eqref{eq:Modified_Curve_Shortening} with fixed endpoints at $p_1,p_2$ which meets no other singularities of $\phi$. Then $\gamma_t$ is strictly convex for all $t\in [0,T]$.
	\end{proposition}
	
	\begin{proof}  Suppose without loss of generality that $\kappa_0>0$ away from $p_1,p_2$.
		Since $\phi^{-1}>0$ away from $p_1,p_2$, we see that in the interior of the arc $\gamma_t$ we have $\kappa_t>0$ if and only if $f_t:=\phi^{-1}\kappa_t>0$.  
		
		We observe that  the flow  $\{\gamma_t\}$ lifts to a smooth Lagrangian mean curvature flow in the hyperk\"ahler space $X$.  The norm  of the projected mean curvature of this flow in $X$ is $\phi^{-1/2}|\kappa_t|$  (cf.~Proposition \ref{prop:curvature_blowup} below).  Since this norm must be bounded for $t\in [0,T]$ we have that $f_t=\phi^{-1}\kappa_t$ vanishes at $p_1,p_2$ for all $t$.  
		
		Notice that we can rewrite \eqref{eq:Evolution_Curvature_Estimate} as (using $'=\partial_s$)
		$$ \partial_t f_t\geq \phi^{-1}f_t''+\kappa_t(\phi^{-1}\kappa_t-2)f_t.$$
		Since we are assuming that $\{\gamma_t\}$ defines a smooth solution to \eqref{eq:Modified_Curve_Shortening} on $[0,T]$ and $\gamma_t$ is a compact curve, we know that the curvature $\kappa_t$ is bounded for all $t\in[0,T]$.  Moreover, the flow $\{\gamma_t\}$ must remain in a compact region of $\mathbb{R}^3$ and so $\phi^{-1}$ is bounded on $\gamma_t$ for all $t\in [0,T]$.  Hence,  there exists some $c>0$ such that
		$$\kappa_t(\phi^{-1}\kappa_t-2)\geq -c\quad\text{for all $t\in[0,T]$}.$$
		Therefore, on the region of space-time where $f_t\geq 0$ (which includes $t=0$ by assumption), we have 
		\begin{equation}\label{eq:Evolution_ft}
			\partial_tf_t\geq \phi^{-1}f_t''-cf_t.
		\end{equation}
		
		Let $\epsilon>0$ and define $$f^{\epsilon}_t=f_t+\epsilon t.$$
		We see from \eqref{eq:Evolution_ft} that
		\begin{align}
			\partial_tf_t^{\epsilon}&\geq \phi^{-1}(f_t^{\epsilon})''-cf_t+\epsilon \nonumber\\
			&=\phi^{-1}(f_t^{\epsilon})''-cf_t^{\epsilon}+\epsilon(ct+1) \nonumber\\
			&>\phi^{-1}(f_t^{\epsilon})''-cf_t^{\epsilon}.\label{eq:Evolution_ft_2}
		\end{align}
		
		Notice that $f_0^{\epsilon}=f_0>0$ away from $p_1,p_2$.   Above we noted that $\{\gamma_t\}$ lifts to a smooth, compact, Lagrangian mean curvature flow $\{L_t\}$ and that $\phi^{-1/2}|\kappa_t|$ is the norm of the projected mean curvature, which is positive at $t=0$ away from $p_1,p_2$.  Since $\{L_t\}$ is smooth, the strong maximum principle implies that the norm of the mean curvature of $L_t$ must be positive for all $t\in(0,\delta)$ for some $\delta>0$.  Therefore, $f_t=\phi^{-1}\kappa_t$ is positive away from $p_1,p_2$ for $t\in (0,\delta)$ as $\phi$  and $\phi^{-1/2}\kappa_t$ are both positive. 
		
		Suppose that $(t_0,x_0)$ is a space-time point away from $p_1,p_2$ where $f_t^{\epsilon}=0$ and $t_0>0$ is minimal.  (Note that any such minimal time must be positive by the argument just given above.)  Then $f_t^{\epsilon}>0$ away from $p_1,p_2$ for all $t<t_0$ and $f_{t_0}^{\epsilon}\geq 0$ with $f_{t_0}(x_0)=0$.  Therefore, we must have that $\partial_tf_t^{\epsilon}\leq 0$ at $(t_0,x_0)$ whereas $x_0$ is a local minimum of $f_{t_0}^{\epsilon}$ and so $(f_t^{\epsilon})''\geq 0$ at $(t_0,x_0)$.  However, we see from \eqref{eq:Evolution_ft_2} that, at $(t_0,x_0)$,
		$$0\geq \partial_tf_t^{\epsilon}>\phi^{-1}(f_t^{\epsilon})''-cf_t^{\epsilon}\geq 0,$$
		which is a contradiction.  
		
		We deduce that 
		$$f_t^{\epsilon}=f_t+\epsilon t>0$$
		away from $p_1,p_2$ for all $t\in [0,T]$ for all $\epsilon>0$.  Letting $\epsilon$ tend to $0$ gives that $f_t$ and hence $\kappa_t$ is non-negative on $[0,T]$.
		
		Now we know that $f_t\geq 0$ everywhere on $[0,T]$ we have that the inequality \eqref{eq:Evolution_ft_2} holds for all $t\in [0,T]$.  This is a parabolic inequality away from $p_1,p_2$ and so if $f_0>0$ away from $p_1,p_2$ then by the strong maximum principle we have that $f_t>0$ away from $p_1,p_2$, which gives the result.
	\end{proof}

	For possible future study we also record the following (easier) convexity result which is immediate from  Proposition \ref{prop:Evolution_Curvature} and the strong parabolic maximum principle.
	
	\begin{proposition}
		Let $\gamma_0\subseteq\mathbb{R}^3$ be an embedded planar curve meeting no singularities of $\phi$ so that the curvature $\kappa_0$ of $\gamma_0$ is non-negative.   Let $\lbrace \gamma_t \rbrace_{t \in [0,T]}$, for $T>0$, be a smooth solution of \eqref{eq:Modified_Curve_Shortening} which meets no other singularities of $\phi$. Then the curvature $\kappa_t$ of $\gamma_t$ is positive for all $t\in [0,T]$.
	\end{proposition}

	\subsection{Evolution of area of bounding holomorphic disks}\label{ss:hol.pacman}  We shall now consider a   special situation which will be of interest in setting up the analysis of singularities. Suppose we have a connected immersed minimal Lagrangian $L_{\infty}$ (in particular, it could be the union of two special Lagrangians with different phases intersecting at a point) and a solution $L_t$ to Lagrangian mean curvature flow, with grading $\theta_t$, which intersects $L_{\infty}$  
	at two points $p_+$ and $p_-$ for all $t$. 
	Let $D$ be the unit disk in $\mathbb{C}$ and write 
	$$\partial D\setminus\{1,-1\}=\partial D^+\sqcup\partial D^-$$ 
	as the disjoint union of two connected components with $D^{\pm}$ contained in the upper/lower half-plane. Suppose further that $\sigma_t:D\to X$ is a family of holomorphic disks with two marked points $\pm 1$ such that
	$$\gamma_t=\sigma_t(\partial D^+)\subseteq L_t,\quad \gamma_{\infty}=\sigma_t(\partial D^-)\subseteq L_{\infty},\quad \sigma_t(\pm 1)=p_{\pm},$$
	where $\gamma_{\infty}$ is independent of $t$.  (We could allow $p_{\pm}$ and $\gamma_{\infty}$ to vary inside $L_{\infty}$ and we will obtain the same answer, but this is not required for our purposes.) 
The situation of particular interest to us is shown in Figure~\ref{fig:evolving.pacman}, where we call $\sigma_t(D)$  a ``holomorphic pacman disk'' (for obvious reasons).  (Note that the central point in Figure~\ref{fig:evolving.pacman} is not a puncture on the boundary of the disk, but rather indicates the immersed point of $L_{\infty}$ there, which is a singularity of $\phi$.)

	
	\begin{figure}[h]
		\begin{center}
			\begin{tikzpicture}
				\fill[fill=yellow] (0,0) arc (240:-60:2);
				\draw[blue,thick] (0,0) arc (240:90:2);
				\draw[blue,thick,->] (2,0) arc (-60:90:2);
				\fill[fill=white] (0,0)--(1,1.73)--(2,0)--(0,0);
				\node (p+) at (0,0) [inner sep=2pt,circle,draw=red,fill=red] {};
				\node (p) at (1,1.73) [inner sep=2pt,circle,draw=red,fill=red] {};
				\node (p-) at (2,0) [inner sep=2pt,circle,draw=red,fill=red] {};
				\node at (-0.5,-0.1) {$p_-$};
				\node at (2.5,-0.1) {$p_+$};
				\draw[black,thick,->] (p+) to node [swap] {} (p);
				\draw[black,thick,->] (p) to node {} (p-);
				\node at (1,0.2) {$\gamma_\infty\subseteq L_{\infty}$};
				\node at (1,2.5) {$\sigma_t(D)$};
				\node at (1,4.2) {$\gamma_t\subseteq L_t$};
			\end{tikzpicture}
		\end{center}
		\caption{Evolving holomorphic pacman disk.}\label{fig:evolving.pacman}
	\end{figure}
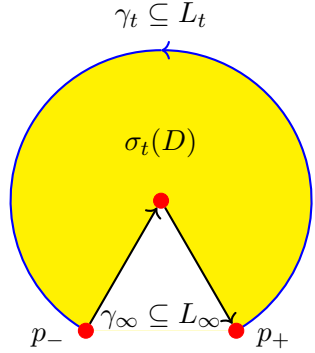

	We may then compute the evolution of the area of the holomorphic disks $\sigma_t(D)$ as follows.
	
\begin{lemma}\label{lem:area.pacman} Using the notation above, the area of the holomorphic disks $\sigma_t(D)$, which is given by	
	\begin{equation}\label{eq:area.holo.disks}
		A(t):=\int_{D} \sigma_t^*\omega,
	\end{equation}
	satisfies
	\begin{align}\label{eq:Area_Hol_Disks_Variation}
		\dot{A}(t) = \theta_t(p_+)-\theta_t(p_-).
	\end{align} 	
	\end{lemma}

\begin{proof}	
	Differentiating \eqref{eq:area.holo.disks} with respect to time yields (using the fact that $d\omega=0$ and $\gamma_\infty$ is fixed)
	\begin{align*}
		\dot{A} & = \int_{D} \sigma_t^*d( \iota_{\partial_t \sigma_t}\omega)  = \int_{\partial D}\sigma_t^*(\iota_{\partial_t \sigma_t}\omega)  = \int_{\gamma_t}\iota_{\partial_t \gamma_t}\omega .
	\end{align*} 
	Recall that $L_t$ is evolving via mean curvature flow 
	$$\partial_tL_t =H_t=J\nabla\theta_t.$$ 
	Therefore, as $\gamma_t\subseteq L_t$, the component of $\partial_t\gamma_t$ normal to $L$ must be $J\nabla\theta_t$.   Moreover, we do not see the component of $\partial_t\gamma_t$ tangential to $L_t$ in $\iota_{\partial_t \gamma_t}\omega$ since $L_t$ is Lagrangian for all $t$.  Therefore,
	$$\iota_{\partial_t \gamma_t} \omega=  \omega(J\nabla \theta_t , \cdot)= -g(\nabla \theta_t , \cdot )=-d\theta_t.$$ By the fundamental theorem of calculus, since we oriented $\gamma_t$ in the  anticlockwise direction so that it is compatible with the orientation induced by $D$, we find that
	\begin{align*} 
		\dot{A} =  
		-\left[ \theta_t(\gamma_t(-1)) - \theta_t(\gamma_t(1)) \right] = \theta_t(p_+)-\theta_t(p_-).
	\end{align*} 
as claimed.
\end{proof}	
	
	\begin{remark}
		In the simplest case, we shall be interested in applying this result  when $L_t$ is Hamiltonian isotopic through circle-invariant Lagrangians to $L_- \# L_+$ for two special Lagrangians $L_-$ and $L_+$ intersecting transversely at a point $p$ and $\sigma_t(D)$ has boundary components on $L_t$ and $L_{\infty}=L_- \cup L_+$, as suggested by Figure~\ref{fig:evolving.pacman}. Notice that only the first of these boundary components needs to be smooth for the computation in the proof of Lemma~\ref{lem:area.pacman}.
	\end{remark}

	\section{Finite time singularities}
	
	This section establishes conditions under which finite time singularities of the circle-invariant Lagrangian mean curvature flow develop and proves, in this setting, a conjecture of Joyce on the local structure of such singularities. This is stated as Theorem \ref{thm:Lawlor}.
	
	Along the way, we shall also prove, in Theorem \ref{thm:Unstable}, that the Lagrangian mean curvature flow starting at a strictly unstable Lagrangian $L=\pi^{-1}(\gamma)$ with $\gamma$ convex develops finite time singularities. This will be improved later in the article by dropping the hypothesis that $\gamma$ be convex.
	
	\subsection{Location of singularities}\label{ss:singularities}
	
	Before we begin we use some key results from \cite{Lotay2020} to provide some important tools for our singularity analysis. 	The first is the following general result.
	
	\begin{proposition}\label{prop:curvature_blowup}  If the curve $\gamma_t\subseteq\mathbb{R}^3$ solves \eqref{eq:Modified_Curve_Shortening} then $\phi^{-1}\kappa_t$ is the projection of the mean curvature of $L_t=\pi^{-1}(\gamma_t)$ to $\mathbb{R}^3$ and the flow exists as long as the norm of the second fundamental form $L_t$ is bounded, which is equivalent to requiring that
		$$\phi^{-1/2}|\kappa_t|\quad\text{and}\quad \phi^{-1/2}|\nabla^{\perp}_{\mathbb{R}^3}\log\phi|$$
		are bounded on $\gamma_t$, where norms and $\nabla^{\perp}_{\mathbb{R}^3}$ are taken with respect to the Euclidean metric on $\mathbb{R}^3$.  Note that $\phi^{-1/2}|\kappa_t|$ is the norm of the projection $\phi^{-1}\kappa_t$ of the mean curvature.
	\end{proposition}
	
	\begin{proof}
		This follows from \cite{Lotay2020}*{Propositions 4.5 and 4.6}, though it should be noted that the factor of $\phi^{-1/2}$ was omitted from the second term there, though this does not affect any of the arguments in that reference.
	\end{proof}

	Using Proposition \ref{prop:curvature_blowup} we have the following theorem based on \cite{Lotay2020}*{Section 6}.
	
	\begin{proposition}\label{prop:blowup_points}  Let $p_1,\ldots,p_k\in\mathbb{R}^3$ for $k>1$ be the singularities of $\phi$.   
		Let $\gamma_0\subseteq\mathbb{R}^3$ be a compact, almost calibrated, planar arc connecting $p_1,p_2$ and meeting no other singularities of $\phi$.  Let $\gamma_t$ be the solution of \eqref{eq:Modified_Curve_Shortening} starting at $\gamma_0$ with the fixed endpoints $p_1,p_2$  and suppose that $\gamma_t$ has a finite time singularity at a point $p$. Then $p=p_i$ for some $i>2$ and $\phi^{-1 }|\kappa_t|$ tends to zero at the singularity while $\phi^{-1/2}|\nabla^{\perp}_{\mathbb{R}^3}\log\phi|$ blows up.
	\end{proposition}

	\begin{proof}  Suppose that $p\neq p_i$ for any $i>2$.  
	By \cite{Lotay2020}*{Lemma 6.6}, the almost calibrated assumption ensures that $|\nabla^{\perp}_{\mathbb{R}^3}\log\phi|$ remains bounded as long as $\gamma_t$ never reaches any other singularities of $\phi$.  Since $\gamma_t$ remains in a bounded region of $\mathbb{R}^3$, $\phi^{-1/2}$ is bounded on $\gamma_t$, and thus  $\phi^{-1/2}|\nabla^{\perp}_{\mathbb{R}^3}\log\phi|$ also remains bounded.  Hence, by Proposition \ref{prop:curvature_blowup}, we must have that $\phi^{-1/2}|\kappa_t|$ blows up at $p$.
	
	However, \cite{Lotay2020}*{Lemmas 6.8, 6.9 and 6.10} can then be used to reach a contradiction (these results do not require the stability assumption, only that $\phi^{-1/2}|\kappa_t|$ blows up).  Specifically, Lemma  6.8 states that $p$ must be either $p_1$ or $p_2$, but then this possibility is ruled out by Lemmas 6.9 and 6.10 which respectively show that in this situation we cannot have $\phi^{-1}|\kappa_t|\to\infty$ at $p_1,p_2$ or $\phi^{-1}|\kappa_t|$ bounded at $p_1,p_2$. This finishes the proof that $p=p_i$ for some $i>2$.
	
We now see that, near $p=p_i$, we have that $\log\phi\sim -\log |x-p|$ and so, for points $\gamma_t(s)$ near $p$,
	\begin{align}\nonumber
		\phi(\gamma_t(s))^{-1/2}|\nabla^{\perp}_{\mathbb{R}^3}\log\phi(\gamma_t(s))| & \sim \frac{1}{|\gamma_t(s)-p|^{1/2}}|\langle \frac{\gamma_t(s)-p}{|\gamma_t(s)-p|},I\gamma_t'(s)\rangle| \\ \label{eq:log.bound}
		& \leq\frac{1}{|\gamma_t(s)-p|^{1/2}}\sim\phi(\gamma_t(s))^{1/2}.
	\end{align}
	with equality if and only if we are at a closest point to $p$ on $\gamma_t$.  Hence $\phi^{-1/2}|\nabla^{\perp}_{\mathbb{R}^3}\log\phi|$ blows up at the finite time singularity at $p=p_i$.

	We now show that $\phi^{-1}|\kappa_t|$ tends to zero at the singularity.  We first note that, since we have a finite-time singularity of mean curvature flow at $p$, the norm of the second fundamental form $|A_{L_t}|$, for $L_t=\pi^{-1}(\gamma_t)$, must blow up as we approach the singularity.  Moreover, since $L_t$ is $S^1$-invariant, if we fix $x\in\gamma_t$ then $|A_{L_t}|(y)$ is constant for $y\in \pi^{-1}(x)$. 
	
	We may therefore take a sequence of spacetime points $(s_i,t_i)$ with $x_i:=\gamma_{t_i}(s_i)\to p$ as $i\to\infty$ so that, for any $y_i\in\pi^{-1}(x_i)$, we have that  $|A_{L_{t_i}}|(y_i)$ is maximized amongst $|A_{L_t}|(y)$ for $y\in L_t$  and $t\leq t_i$.   Proposition \ref{prop:curvature_blowup} shows that  
	$$|A_{L_{t_i}}| \sim \phi^{-1/2}(\gamma_{t_i}) |\kappa_{t_i}| + \phi^{-1/2}(\gamma_{t_i}) |\nabla_{\mathbb{R}^3}^\perp \log (\phi(\gamma_{t_i})) |. $$ 
From this and \eqref{eq:log.bound} we deduce that, for $i$ large,
\begin{equation}\label{eq:sff.sim}
\sup  |A_{L_{t_i}}| \sim |x_i-p|^{1/2} |\kappa_{t_i}| + |x_i-p|^{-1/2}.
\end{equation}
	Furthermore, by Proposition \ref{prop:curvature_blowup}, the first term on the right-hand side of \eqref{eq:sff.sim} is the contribution of the mean curvature $H_{L_{t_i}}$ of $L_{t_i}$.
	
	We now suppose, for a contradiction, that there is $\delta>0$ such that for some subsequence, that we also denote by $(s_i,t_i)$, we have  
	\begin{equation}\label{eq:not.zero}\phi^{-1}(\gamma_{t_i}(s_i)) |\kappa_{t_i}(s_i)| \geq 2\delta \quad\text{for all $i$.}
	\end{equation}
 This implies that  
	$$|x_i-p| |\kappa_{t_i}(s_i)| \geq \delta ,$$
	for all $i$ sufficiently large,  
	which we can equivalently write as
\begin{equation}\label{eq:sff.est}|x_{i}-p|^{1/2} |\kappa_{t_i}| \geq \delta |x_{i}-p|^{-1/2}. 
\end{equation}
	Putting together \eqref{eq:sff.sim}, \eqref{eq:sff.est} and Proposition \ref{prop:curvature_blowup}, we can bound the second fundamental form of $L_{t}$ in terms of its mean curvature $H_{L_{t}}$:
	\begin{align}
\sup_{t\leq t_i}|A_{L_t}|= \sup |A_{L_{t_i}}| & \sim |x_i-p|^{1/2} |\kappa_{t_i}| + |x_i-p|^{-1/2} \nonumber\\
		& \lesssim (1+\delta^{-1}) |x_i-p|^{1/2} |\kappa_{t_i}| \nonumber\\
		& \lesssim  (1+\delta^{-1}) \sup |H_{L_{t_i}}| \leq (1+\delta^{-1})\sup_{t\leq t_i} |H_{L_t}|.\label{eq:sff.bound.H}
	\end{align}
	
	Recall that $|A_{L_t}|$ is maximised for $t\leq t_i$ at any $y_i\in \pi^{-1}(x_i)$.  We may therefore consider the type II blow up of $\lbrace L_{t} \rbrace$ in $X$ centered at spacetime points $(y_i,t_i)$ with $y_i\in \pi^{-1}(x_i)$.  To this end, we set $q = \pi^{-1}(p)$, which is a point, and identify a fixed  ball $B$ around $q$ in $X$ with a ball around the origin in $\mathbb{C}^2$ of the same radius, which we also denote $B$.  Then, since $y_i \to q$ as $i\to\infty$, we have that $y_i \in B$ for all $i$ sufficiently large. We may therefore consider the sequence
	$$L^{i}_s = \lambda_i (L_{t_i + \lambda_i^{-2} s}-q_i)  \ \ \text{for $s \in (-\lambda_i^2 t_i, 0)$,}$$
	with $\lambda_i = |A_{L_{t_i}}|(y_i)$.  Note that $|A_{L^i_s}|$ is uniformly bounded above by $1$, and attains the value $1$ at the origin at $s=0$. Therefore, after passing to a subsequence, we may take a limit as $i \to  \infty$ to obtain convergence of $\{L^i_s\}$ to an ancient solution to Lagrangian mean curvature flow in $\mathbb{C}^2$, which we denote by $\lbrace L^{\infty}_s \rbrace_{s \in (-\infty ,0)}$.  
We now observe that the bound \eqref{eq:sff.bound.H} is scale-invariant, so holds on $L^i_s$, and recall that $\sup_{s<0}|A_{L^i_s}|=1$.  Taking limits, we deduce that 
	\begin{equation}\label{eq:Mean Curvature Bound}
		1 = \sup_{s<0} |A_{L^\infty_{s}}| \lesssim (1+\delta^{-1}) \sup_{s<0} |H_{L^\infty_{s}}|.
	\end{equation}

The ancient solution $\{L^{\infty}_s\}$ is exact and almost calibrated, since the original flow $\{L_t\}$ in $X$ had these properties. Furthermore, because $L_{t}=\pi^{-1}(\gamma_t)$ is almost calibrated, we have that the $\gamma_{t_i}$ are graphical over a straight line and thus, after blowing up, the projection of $L^{\infty}_s)$ to $\mathbb{R}^3$ must be asymptotic to at most two multiplicity one lines, since a multiplicity two line is ruled out by the almost calibrated condition.  Thus, $L^\infty_{s}$ is asymptotic to at most two multiplicity one planes. In \cite[Theorem 1.1]{LambertLotaySchulze} such  exact, almost calibrated, ancient solutions to Lagrangian mean curvature flow in $\mathbb{C}^2$ are classified:  these are planes and Lawlor necks.  In either case $L^{\infty}_s$ is minimal (and hence a static solution), so $H_{L^{\infty}_s}=0$, but this violates \eqref{eq:Mean Curvature Bound}, giving our desired contradiction to \eqref{eq:not.zero}.
	\end{proof}

	We can also deal with the case of some non-compact curves as follows.
		
	\begin{proposition}\label{prop:blowup_points.2}
		Suppose that $p_1,\dots,p_k\in\mathbb{R}^3$ for $k\geq 1$ are the singularities of $\phi$ and let $\ell_+,\ell_-$ be rays starting at $p_1$ and meeting no other singularities of $\phi$.  Let $\gamma_0\subseteq\mathbb{R}^3$ be an almost calibrated   planar arc which lies in the same plane as $\ell_-\cup\ell_+$ and is asymptotic to $\ell_{-}\cup\ell_+$ at infinity in the sense that, outside a compact set, $\gamma_0$ is a smooth graph of a function $u$ over $\ell_-\cup\ell_+$ so that $|u|\to 0$ at infinity.  Suppose further that $\gamma_0$ meets no singularities of $\phi$.  
		
		There is a unique short-time solution $\gamma_t$ of \eqref{eq:Modified_Curve_Shortening} starting at $\gamma_0$ which remains asymptotic to $\ell_{-}\cup\ell_+$ at infinity.  Moreover, if $\gamma_t$ has a finite time singularity at a point $p$ then $p=p_i$ for some $i$, $\phi^{-1}|\kappa_t|$ tends to zero at the singularity but $\phi^{-1/2}|\nabla^{\perp}_{\mathbb{R}^3}\log\phi|$ blows up.
	\end{proposition}
	
	\begin{proof}
		Since there are only finitely many singularities of $\phi$, outside of some compact set, the flow \eqref{eq:Modified_Curve_Shortening} is uniformly equivalent to the usual curve shortening flow.  Hence, we may apply the theory of the curve shortening flow, including pseudolocality, to deduce that a unique solution $\gamma_t$  to \eqref{eq:Modified_Curve_Shortening} exists and remains asymptotic to $\ell_-\cup\ell_+$ at infinity.
		
		The proof now proceeds exactly as for Proposition \ref{prop:blowup_points} because all the analysis of finite time singularities is local and there are no finite time singularities outside of a compact set in $\mathbb{R}^3$.
	\end{proof}
	
	\begin{remark}
		The existence and uniqueness of the flow $\gamma_t$ in Proposition \ref{prop:blowup_points.2} can also be deduced by considering the Lagrangian mean curvature flow $L_t=\pi^{-1}(\gamma_t)$ and using methods from \cite{WeiBoSu}. 
	\end{remark}

	\subsection{Neck pinches}
We now turn to the local structure at the finite time singularity of Lagrangian mean curvature flow as considered thus far in this article. We shall use the results  in the previous subsection to show that by rescaling up a finite size neighbourhood of the singular point, the Lagrangian mean curvature flow approaches a fixed Lawlor neck. Consequently, the family of rescaled Lawlor necks gives a first order approximation for the evolution of the flow as it develops a finite time singularity. This result, given in the two theorems below, proves  Theorem \ref{thm:Lawlor_Intro}(a).  Recall the notation that $X$ is a hyperk\"ahler 4-manifold with a circle action as described in Subsection \ref{ss:notation}.
	
	\begin{theorem}
	\label{thm:Lawlor} Let $L$ be an embedded, almost calibrated, circle-invariant Lagrangian in $X$ which is either compact or asymptotic at infinity to a pair of planes.
		Suppose that $\lbrace L_t \rbrace_{t \in [0,T)}$ is an embedded, almost calibrated, circle-invariant solution to Lagrangian mean curvature flow in $X$ starting at $L$, that is also either compact or asymptotic to pair of planes respectively, which develops a finite time singularity at $p \in X$ when $t \to T<\infty$. 

	Then,    for any sequence of times $t_i
 \nearrow T$ as $i\to\infty$, 
	after passing to a subsequence which we also call $t_i$, there are: 
		\begin{itemize}
			\item open neighbourhoods $U$ of $p$ in $X$ and $V$ of $0$ in $T_pX \cong \mathbb{C}^2$;
			\item a pointed isomorphism $\varphi:U \to V$ at $p$; 
\item a nullsequence $\epsilon_i \searrow 0$
		\end{itemize}
such that $ \epsilon_i^{-1}\varphi(L_{t_i} \cap U )$ converges on compact subsets of $\mathbb{C}^2$ to a Lawlor neck $\hat{L}=\pi_H^{-1}(\hat{\gamma})$, 
where $\pi_H:\mathbb{C}^2\to\mathbb{R}^3$ is the radially extended Hopf fibration and $\hat{\gamma}$ is a straight line at distance $1$ from the origin.
	\end{theorem}
	
	\begin{proof} By assumption we know that $L_t=\pi^{-1}(\gamma_t)$ where $\{\gamma_t\}_{t\in[0,T)}$ is a family of planar curves in $\mathbb{R}^3$ which is a solution to \eqref{eq:Modified_Curve_Shortening}  of the type considered in Proposition \ref{prop:blowup_points} or \ref{prop:blowup_points.2}.
	Hence, if $\{\gamma_t\}_{t\in [0,T)}$ develops a finite time singularity at time $T$, this must occur at a point $p=p_i \in \mathbb{R}^3$ for $i>2$ and $\phi^{-1/2}|\nabla_{\mathbb{R}^3}^{\perp} \log \phi|$ blows up there whilst $\phi^{-1}|\kappa_t|$ stays bounded.  It is also useful to recall the estimate \eqref{eq:log.bound} for   $\phi^{-1/2}|\nabla_{\mathbb{R}^3}^{\perp} \log \phi|$.  
	
	Parametrize the curves $\lbrace \gamma_t \rbrace_{t \in [0,T)}$ using their respective arc-length parameters and take a sequence of space-time points $(s_i, t_i)$ such that $t_i \nearrow T$ as $i\to\infty$ and 
\begin{equation}\label{eq:lambda.i}
\lambda_i := |\gamma_{t_i}(s_i)-p|^{-1} = \max \lbrace  |\gamma_{t}(s)-p|^{-1} \  :  \ s \geq 0 , \ t \leq t_i \rbrace . 
\end{equation}
	In particular, $\gamma_{t_i}(s_i)$ is a closest point to $p$ on $\gamma_{t_i}$.  
	We have that $\lambda_i \nearrow + \infty$ and so we can find $c>0$ such that the ball $B_c(p)\subseteq\mathbb{R}^3$ has the property that $\gamma_{t_i}\cap B_c(p)$ is non-empty for all $i$ sufficiently large and the metric on $\pi^{-1}(B_c(p))$ is close to the Euclidean metric on $\mathbb{R}^4$. The latter property of $B_c(p)$ is equivalent to saying that $2\phi(x)\sim |x-p|^{-1}$ on $B_c(p)$.  
	
	For $t\in[-\lambda_i^2t_i,0)$  
	and $i$ sufficiently large we consider the following curves on $B_{c\lambda_i}(0)\subseteq\mathbb{R}^3$:
	\begin{equation}\label{eq:hat.gamma}
		\hat{\gamma}_t^i := \lambda_i (\gamma_{t_i + \lambda_i^{-2}t}-p). 
	\end{equation}
We denote the corresponding Lagrangians in $\mathbb{C}^2$ by $\hat{L}_t^i= \pi_H^{-1}(\hat{\gamma_t}^i)$.  
Notice that we are blowing up balls centered at $p$ rather than balls centered at $\gamma_{t_i}(s_i)$.  

Observe that  the curvature $\hat{\kappa}^i_t$ of $\hat{\gamma}_t^i$ satisfies
\begin{equation}\label{eq:hat.kappa.i}
|\hat{\kappa}^i_t|=\lambda_i^{-1}|\kappa_{t_i+\lambda_i^{-2}t}|\sim \frac{|\gamma_{t_i}(s_i)-p|}{|\gamma_{t_i+\lambda_i^{-2}t}-p|}\phi(\gamma_{t_i+\lambda_i^{-2}t})^{-1}|\kappa_{t_i+\lambda_i^{-2}t}| 
\end{equation}
by definition of $\lambda_i$ in \eqref{eq:lambda.i} and our assumptions which ensure that $2|\gamma_{t_i+\lambda_i^{-2}t}-p|\phi(\gamma_{t_i+\lambda_i^{-2}t})$ is approximately $1$ on the ball $B_c(p)$.  By choice of $\gamma_{t_i}(s_i)$ in \eqref{eq:lambda.i} we see that  the quotient on the right-hand side of \eqref{eq:hat.kappa.i} is bounded above by $1$, since $t<0$.  By Propositions \ref{prop:blowup_points} and \ref{prop:blowup_points.2}, we deduce from \eqref{eq:hat.kappa.i} that for all $i$ sufficiently large we have that $|\hat{\kappa}_i^t|$ is bounded above by a uniform constant and tends to $0$ as $i\to\infty$.  

As a result, after passing to a subsequence, we may extract a smooth limit $\hat{\gamma}^{\infty}_t$ of the sequence $\hat{\gamma}_t^i$ of \eqref{eq:hat.gamma} which must be a smooth embedded curve with curvature $0$ and so is a single straight line $\hat{\gamma}$.  Furthermore, under the translation and scaling of $B_c(p)$, we see that $p$ gets mapped to $0$ and we let $y_i$ be the image of $\gamma_{t_i}(s_i)$. Then, 
$$|y_i|=\lambda_i^{-1}|\gamma_{t_i}(s_i) - p| =1$$ 
and by definition of $(s_i,t_i)$ in \eqref{eq:lambda.i} we find that $y_i$ is the closest point to the origin in $\hat{\gamma}^i_t$.  Therefore, $\hat{\gamma}$ is a straight line at distance $1$ from the origin.  We deduce that $\hat{L}=\pi_H^{-1}(\hat{\gamma})$ is a Lawlor neck, which is the limit of the sequence $\hat{L}_t^i$.  The result then follows.	
	\end{proof}
	
	We can now use work from \cite{FSS.neck} to improve Theorem \ref{thm:Lawlor} as follows.
	
	\begin{theorem}\label{thm:Lawlor+}
	Let $L\subseteq X$, $\{L_t\}_{t\in [0,T)}$, $p\in X$ and $\pi_H:\mathbb{C}^2\to\mathbb{R}^3$ be as in Theorem \ref{thm:Lawlor}.  There are
	\begin{itemize}
\item open neighbourhoods $U$ of $p$ in $X$ and $V$ of $0$ in $T_pX\cong\mathbb{C}^2$;
\item a pointed isomorphism $\varphi:U\to V$ at $p$;
\item a small $\delta>0$ and a smooth function $\epsilon:(T-\delta^2,T)\to (0,\delta)$, with $\epsilon(t)\searrow 0$ as $t\nearrow T$,	
\end{itemize}
such $\epsilon(t)^{-1}\varphi(L_t\cap U)$ converges on compact subsets of $\mathbb{C}^2$ to a unique Lawlor neck $\hat{L}=\pi_H^{-1}(\hat{\gamma})$, where $\hat{\gamma}$ is a unique straight line at distance $1$ from the origin.
	\end{theorem}
	
	\begin{proof}
	Theorem \ref{thm:Lawlor} shows that at least one tangent flow at the singularity at $p$ is a special Lagrangian union of two transverse planes in $\mathbb{C}^2$, which are the asymptotics of the Lawlor neck $\hat{L}$   in the statement.  Noting that $L_t$ is almost calibrated, and thus zero Maslov, and also exact for all $t$, this is precisely a finite time singularity of Lagrangian mean curvature flow as studied in \cite{FSS.neck}.
	
As stated, the results of \cite{FSS.neck} only apply to Lagrangian mean curvature flow in $\mathbb{C}^2$ or in a compact Calabi--Yau 2-fold.  Since our analysis of the   flow here takes place solely within a fixed small neighbourhood of $p$ in $X$, the analysis from \cite{FSS.neck} carries over to this setting.  

In particular, \cite{FSS.neck}*{Theorem 8.2} shows that the tangent flow at $p$ is in fact unique.  Then \cite{FSS.neck}*{Theorem 8.3} shows that for all times $t$ near $T$ we can find the scalings $\epsilon(t)$  so that the rescaled flow $\epsilon(t)^{-1}\varphi(L_t\cap U)$ is a small $C^1$ graph over a unique Lawlor neck of a given scale.  This Lawlor neck must therefore be $\hat{L}$ given in Theorem \ref{thm:Lawlor}.  

The only difference between the statements in \cite{FSS.neck}*{Theorem 8.3} and the one claimed is that we work with the fixed ball centred at $p$, rather than working with balls with different centres.  This can be achieved since it can be done for one sequence of times $t_i\nearrow T$ by Theorem \ref{thm:Lawlor}.  
	\end{proof}
	
	\begin{remark}
	Theorem \ref{thm:Lawlor+} shows that a unique Lawlor neck $\hat{L}$ models $\lbrace L_t \rbrace_{t \in [0,T)}$ in a fixed size neighbourhood of $p$ after rescaling by $\epsilon(t)^{-1}$.  The neighbourhood in question is $\pi^{-1}(B_{c}(p))$, where $c>0$ is fixed and $\pi : X \to \mathbb{R}^3$ is the hyperk\"ahler moment map, which is locally modelled on $\pi_H$  near $p$. 
	Alternatively, we can view $\varphi^{-1}(\epsilon(t)\hat{L}\cap V)$ as modelling $L_t$ on $U$, i.e.~the scaled down Lawlor necks give an approximation to the flow, once we identify a neighbourhood of $0$ in $\mathbb{C}^2$ with a neighbourhood of $p$ in $X$.   
\end{remark}

	\subsection{Finite time singularities and the area of pacman disks}
	
	In this subsection we show that strictly unstable Lagrangians develop finite time singularities under Lagrangian mean curvature flow.  The main part of the argument uses the maximum principle and a barrier construction. To formulate the barriers, it is useful to introduce the following concepts. 
	
	\begin{definition}\label{def:Triad}
		Let $p_-, p_+,p$ be singularities of $\phi$ lying in a plane and consider a configuration of two oriented embedded planar curves $\gamma_-$, $\gamma_+$ respectively starting and ending at $p_-$, $p$ and $p$, $p_+$, such that $\gamma_{\pm}$ are unions of straight lines connecting singularities of $\phi$.  Let $\ell$ denote the oriented straight line from $p_-$ to $p_+$. If we have an oriented embedded curve $\gamma$ starting at $p_-$ and ending at $p_+$ so that the interior of the region bounded by $\gamma\cup \gamma_-\cup\gamma_+$ is connected and $p$ lies in the interior of the region bounded by $\gamma\cup\ell$,  we shall call the triple $( \gamma_- , \gamma_+ , \gamma )$ a \emph{triad with vertices $(p_- , p_+,p)$}: see Figure \ref{fig:pacman}.

	Given such a triad, we let $\ell_-,\ell_+$ denote the oriented straight lines from $p_-$ to $p$ and $p$ to $p_+$ respectively and 
		let $\theta_{\pm}$ denote the angles made by $\ell_{\pm}$ with $\ell$.  In Figure \ref{fig:pacman} we have that $\ell_{\pm}=\gamma_{\pm}$.
	\end{definition}
	
	\vspace{-5pt}
	
				\begin{figure}[h]
		\centering
		\begin{center}
			\begin{tikzpicture}
				\fill[fill=yellow] (0,0) arc (240:-60:2);
				\draw[blue,thick] (0,0) arc (240:90:2);
				\draw[blue,thick,->] (2,0) arc (-60:90:2);
				\fill[fill=white] (0,0)--(1,1.73)--(2,0)--(0,0);
				\node (p+) at (0,0) [inner sep=2pt,circle,draw=red,fill=red] {};
				\node (p) at (1,1.73) [inner sep=2pt,circle,draw=red,fill=red] {};
				\node (p-) at (2,0) [inner sep=2pt,circle,draw=red,fill=red] {};
				\node at (-0.5,-0.1) {$p_-$};
				\node at (2.5,-0.1) {$p_+$};
				\draw[black,thick,->] (p+) to node [swap] {} (p);
				\draw[black,thick,->] (p) to node {} (p-);
				\node at (1,2) {$p$};
				\node at (2,0.5) {$\gamma_+$};
				\node at (0,0.5) {$\gamma_-$};
				\node at (1,4.2) {$\gamma$};
				\draw[thick,->](p+) to node {} (p-);
				\node at (1,-0.2) {$\ell$};
				\draw[thick,dashed] (p) to (2,1.73);
				\node at (0.5,0.2) {$\theta_-$};
				\draw (0.8,0) arc (0:60:0.8);
				\node at (1.5,1.48) {$\theta_+$};
				\draw (1.8,1.73) arc (0:-60:0.8);
			\end{tikzpicture}
		\end{center}
	 
\vspace{-10pt}		
		
		\caption{The triad and its associated pacman disk.}\label{fig:pacman}
	\end{figure}
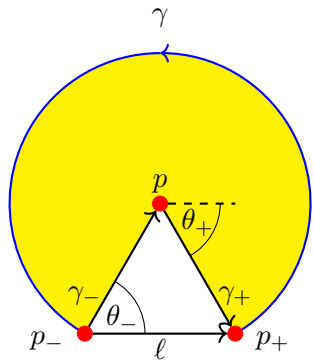 
	
	\begin{definition}\label{def:pacman} 
		Let $(\gamma_-,\gamma_+,\gamma)$ be a triad with vertices $(p_-,p_+,p)$ as in Definition \ref{def:Triad}. 	
		The interior of the planar region enclosed by $\gamma$ and $\gamma_{\infty}=\gamma_- \cup \gamma_+$ will be called a \emph{pacman disk} associated with the triad: see Figure \ref{fig:pacman}.
	\end{definition}

	We shall now show how the area of a pacman disk varies when the curve $\gamma$ in the triad evolves via the flow \eqref{eq:Modified_Curve_Shortening}, utilizing Lemma \ref{lem:area.pacman}.

	\begin{proposition}\label{prop:Finite_Time_Singularity_Area}
		Let $( \gamma_- , \gamma_+ , \gamma_0 )$ be a triad with vertices $(p_-,p_+,p)$ and let $\gamma_t$ evolve via the flow \eqref{eq:Modified_Curve_Shortening}.  Let $\theta_{\pm}$ be as in Definition \ref{def:Triad}.
		
Then $(\gamma_-,\gamma_+,\gamma_t)$ remains a triad with vertices $(p_-,p_+,p)$ as long the flow does not meet another singularity of $\phi$. If $A(t)$ is the area of the pacman disk associated with this triad as in Definition \ref{def:pacman}, we have
\begin{equation}\label{eq:dot.A.ineq}
\dot{A} \leq \theta_+ - \theta_-.
\end{equation}
		In particular, if $\theta_+ < \theta_-$ then $A(t) \leq A(0)-(\theta_- - \theta_+)t$ and so $\gamma_t$ meets $p$ in finite time $T\leq A(0)/(\theta_- - \theta_+)$, if it does not reach some other singularity of $\phi$ before $t=T$.
	\end{proposition}
	\begin{proof}
We first observe that the conditions on $(\gamma_-,\gamma_+,\gamma_0)$ being a triad mean that $\gamma_0$ does not intersect $\gamma_{\infty}=\gamma_-\cup\gamma_+$ except at $p_-,p_+$.  Since $\gamma_\infty$ is a union of straight lines, it remains stationary along the flow \eqref{eq:Modified_Curve_Shortening}.  Hence, by the maximum principle, $\gamma_\infty$  acts as a barrier along the flow, except at any points along $\gamma_\infty$ which are singularities of $\phi$. Thus, under the assumption that the flow does not meet a singularity of $\phi$, $(\gamma_-,\gamma_+,\gamma_t)$ remains a triad with vertices $(p_-,p_+,p)$.  Note that a singularity cannot develop at $p_-,p_+$ by the work in \cite{Lotay2020}*{$\S$6}, where only the graded assumption (which is possible to make since $\gamma_t$ is embedded) is used to rule out singularities at the endpoints.
		
		Since the pacman disks associated with the triads $(\gamma_-,\gamma_+,\gamma_t)$ must lie in a fixed plane for all time, we may assume without loss of generality they lie in the $(\mu_2,\mu_3)$-plane.  We may then lift these pacman disks to holomorphic pacman disks in $X$ as in $\S$\ref{ss:hol.pacman}, with boundary on curves  contained in $L_t=\pi^{-1}(\gamma_t)$ and $L_{\infty}=\pi^{-1}(\gamma_{\infty})$ that project to $\gamma_t$ and $\gamma_{\infty}$ under $\pi$.  Since $\gamma_\infty$ is a connected union of straight line segments, $L_{\infty}$ is a connected immersed minimal Lagrangian, and $L_t$ satisfies Lagrangian mean curvature flow as $\gamma_t$ satisfies \eqref{eq:Modified_Curve_Shortening}.

We may therefore apply Lemma \ref{lem:area.pacman} to give that
		\begin{equation}\label{eq:area.pacman}
			\dot{A}(t)  = \theta_t(p_+) - \theta_t(p_-) .
		\end{equation}
Since $\gamma_{\infty}$ is a barrier for the flow $\gamma_t$, as we argued above, the lines $\ell_{\pm}$ are parts of this barrier.
 Therefore, there is a constant $c \in \mathbb{R}$ (which is a multiple of $\pi$) such that $\theta_t(p_+) \leq \theta_+ + c$ and $\theta_t(p_-) \geq \theta_- + c$, since this is true initially (i.e.~for $t=0$) by the conditions on the triple $(\gamma_-,\gamma_+,\gamma_0)$ being a triad with vertices $(p_-,p_+,p)$ as in Definition \ref{def:Triad}.  The inequality \eqref{eq:dot.A.ineq} then follows from \eqref{eq:area.pacman}.
 The final result is then an easy consequence of \eqref{eq:dot.A.ineq}. 
	\end{proof}

A consequence of Proposition \ref{prop:Finite_Time_Singularity_Area} is that we may now prove most of part (b) of Theorem \ref{thm:Lawlor_Intro}. 

	\begin{theorem}\label{thm:Finite_Time}
		Let $X \neq \mathbb{R}^3 \times \mathbb{S}^1$ be an ALE or ALF hyperk\"ahler $4$-manifold constructed via the Gibbons--Hawking ansatz. Then, there is an almost calibrated, embedded, circle-invariant Lagrangian $L_0 \subset X$, diffeomorphic to $\mathbb{S}^1\times\mathbb{R}$,  such that the Lagrangian mean curvature flow starting at $L_0$ develops a finite time singularity as in Theorem \ref{thm:Lawlor}.  
	\end{theorem}
	\begin{proof}
		If $X \neq \mathbb{R}^3 \times \mathbb{S}^1$ then $\phi$ has at least one singularity which, with no loss of generality, we suppose to be located at the origin in $\mathbb{R}^3$. Then, we define two infinite non-collinear planar rays $\gamma_-, \gamma_+$ emanating from $0$ which meet no other singularities of $\phi$. We orient these rays so that $\gamma_-$ is directed towards $0$ and $\gamma_+$ is oriented towards its noncompact end. Notice that both $\pi^{-1}(\gamma_{\pm})$ are special Lagrangian $\mathbb{R}^2$'s in $X$. Then, we set $\gamma_0 \subset \mathbb{R}^3$ to be an infinite planar curve asymptotic to $\gamma_-$ and $\gamma_+$ at its two ends such that: it goes round 0 to connect its two ends through the side of $0$ that makes an angle $\alpha > \pi$; and the area, $A$, of the infinite region enclosed by $\gamma_0 \cup \gamma_- \cup \gamma_+$ is finite. This may be thought of as the triad depicted in Figure \ref{fig:pacman} with the points $p_{\pm}$ sent to infinity. Then, the same argument as in Proposition \ref{prop:Finite_Time_Singularity_Area} shows that
		$$\dot{A}\leq  \theta_+ - \theta_-,$$
		where $\theta_{\pm}$ are the angles made by $\gamma_{\pm}$ with a fixed line in the plane. We may always choose these rays so that $\theta_->\theta_+$.  Then, the Lagrangian mean curvature flow starting at $L_0=\pi^{-1}(\gamma_0)$ must reach $0$ in a finite time less than $A(0)/(\theta_- - \theta_+)$ and so must develop a finite time singularity at least by that time.  Note that $L_0$ is topologically $\mathbb{S}^1\times\mathbb{R}$ by construction.
	\end{proof}	
	
\begin{remark}		In \cite[Problem 3.12(a)]{JoyceConjectures}, Joyce asks
for examples of Lagrangian mean curvature flows which develop Lawlor neck pinch singularities;  Theorem \ref{thm:Finite_Time} provides a large class of such examples
in dimension 4. As suggested in \cite[Problem 3.12(b)]{JoyceConjectures}, it would be
interesting to study the stability of this kind of singularity formation
under small (non-circle-invariant) Hamiltonian deformations of
our Lagrangians.
\end{remark}
	
	\subsection{Unstable strictly convex curves}
	
	Proposition \ref{prop:Finite_Time_Singularity_Area} has an important consequence: namely that circle-invariant strictly unstable Lagrangians develop finite time singularities. Here we shall not yet prove this result in its full generality but a simplified version for a certain class of strictly unstable circle-invariant Lagrangians $L=\pi^{-1}(\gamma)$ for which $\gamma$ is strictly convex. 
	We shall see that, as a consequence, we can complete the proof of Theorem \ref{thm:Lawlor_Intro}(b). 
	Recall the notation that $X$ is a hyperk\"ahler 4-manifold with a circle action as described in Subsection \ref{ss:notation}.
	\begin{theorem}
	\label{thm:Unstable}
		Let $L=\pi^{-1}(\gamma)$ be a  compact, embedded, almost calibrated, circle-invariant Lagrangian in $X$. If $\gamma$ is strictly convex and strictly unstable, then the Lagrangian mean curvature flow $L_t=\pi^{-1}(\gamma_t)$ starting at $L_0=L$ attains a finite time singularity at $\pi^{-1}(p)$ for some singularity $p$ of $\phi$ in the region bounded by $\gamma$ and the straight line connecting its endpoints.
	\end{theorem}
	\begin{proof}
		From the assumptions on $L$ and $\gamma$, we know that there are singularities $p_-,p_+$ of $\phi$ so that $\gamma$ is a planar embedded arc from $p_-$ to $p_+$ and that $\gamma$ meets no other singularities of $\phi$.  
		If $\gamma$ is strictly convex then, using the notation of Definition \ref{def:Triad},   the region $\Delta$ bounded by $\gamma\cup\ell$ is convex.  
		
		As $\gamma$ is strictly unstable, there must be singularities of $\phi$ in the interior of $\Delta$  which are not equal to $p_-,p_+$; let $\Delta_{\infty}\subseteq\Delta$ be the convex hull of the union of these singularities with $p_-,p_+$.  Then $\Delta_{\infty}$ is a convex polygon with boundary given by a union of straight lines connecting singularities of $\phi$, with one side given by $\ell$.  
		
		Let $p$ be any singularity of $\phi$ on the boundary of $\Delta_{\infty}$ not equal to $p_-,p_+$ and let $\gamma_-$ be the oriented union of the sides of $\Delta_{\infty}$ connecting $p_-$ to $p$, and let $\gamma_+$ be the oriented union of the sides of $\Delta_{\infty}$ connecting $p$ to $p_+$.  By construction $(\gamma_-,\gamma_+,\gamma)$ is a triad with vertices $(p_-,p_+,p)$ and since $\Delta_{\infty}$ is convex we know, in the notation of Definition \ref{def:Triad}, that $\theta_+<\theta_-$.  
		
		Proposition \ref{prop:convexity} states that the solution $\gamma_t$ to \eqref{eq:Modified_Curve_Shortening} remains strictly convex, as well as almost calibrated and embedded, as long as the flow does not meet any other singularities of $\phi$.  Therefore, the flow is pointing into the region $\Delta$ for all $t$ for which the flow exists smoothly.  We deduce that $\gamma_t$ lies in $\Delta$ and $(\gamma_-,\gamma_+,\gamma_t)$ is a triad with vertices $(p_-,p_+,p)$ for all $t$ before $\gamma_t$ reaches a singularity of $\phi$ by Proposition \ref{prop:Finite_Time_Singularity_Area}.
		
		Proposition \ref{prop:Finite_Time_Singularity_Area} then implies that $\gamma_t$ reaches $p$ in finite time unless it reaches some other singularity of $\phi$, which must be in $\Delta$, before then. Since there are only a finite number of singularities of $\phi$ on the boundary of $\Delta_{\infty}$, there must be a first finite time $T$ so that $\gamma_t$ reaches a singularity (which we now call $p$) of $\phi$ not equal to $p_{\pm}$ lying on $\partial \Delta_{\infty}$.
		
		Since $L_T$ would  be given topologically by the union of at least two spheres meeting at a point, whereas $L_t$ is a single sphere for $t<T$, the flow must have a singularity at $p$ at time $t=T$.
		\end{proof}

An immediately corollary of Theorem \ref{thm:Unstable} is the following result which guarantees the existence of a compact Lagrangian whose Lagrangian mean curvature flow develops a neck-pinch singularity, and completes the proof of Theorem \ref{thm:Lawlor_Intro}(b).
	
	\begin{theorem}\label{thm:Finite_Time_2}
		Let $X$ be an ALE or ALF hyperk\"ahler 4-manifold constructed via the Gibbons--Hawking ansatz so that $\phi$ has at least three singular points that lie in a plane but which are not collinear. Then, there is a compact, almost calibrated, embedded, circle-invariant Lagrangian $L_0 \subset X$, diffeomorphic to $\mathbb{S}^2$, such that the Lagrangian mean curvature flow starting at $L_0$ develops a finite time singularity as in Theorem \ref{thm:Lawlor}.   
	\end{theorem}
	
	\begin{proof}
			If we have the  assumption that $\phi$ has three coplanar but not collinear singularities, we may label them $p_-,p_+,p$ and arrange them in a triangle as in Figure \ref{fig:pacman}. Let $\ell$ be the straight line from $p_-$ to $p_+$  We may then clearly choose a strictly convex curve $\gamma$ connecting $p_-$ and $p_+$, which lies in the same plane as $p_-,p_+,p$, meets no other singularities of $\phi$ and so that the region bounded by $\gamma\cup\ell$ contains $p$. Since $p_-,p_+,p$ are not collinear, we deduce that $\gamma$ is strictly unstable, and so we may apply Theorem \ref{thm:Unstable}.  Note that in this case, by the construction of $\gamma$, we have that $\pi^{-1}(\gamma)$ is a 2-sphere.
	\end{proof}

	\section{Flow through singularities and long--time behaviour}
	
	In this section we shall prove that the Lagrangian mean curvature flow starting at an embedded, almost calibrated, circle invariant Lagrangian exists and can be continued through its finite time singularities, giving rise to a flow that exists for all time. Furthermore, we prove that at infinite time, such a Lagrangian mean curvature flow with surgeries converges to a union of special Lagrangian submanifolds. The main result is stated as Theorem \ref{thm:Flow}.

	\subsection{Flow of piecewise smooth curves}
	
	In this subsection we shall prove that for the modified curve shortening flow \eqref{eq:Modified_Curve_Shortening} for suitable planar curves $\gamma_t \subset P\cong \mathbb{R}^2  \subset \mathbb{R}^3$, we can flow through each finite time singularity which occurs. This flow of curves through singularities then gives rise to a Lagrangian mean  curvature flow $L_t:=\pi^{-1}(\gamma_t)$ in the total space of $\pi: X \to \mathbb{R}^3$. 
	
To state the properties of this flow with surgeries we make the following definition.

	\begin{definition}\label{def:Flow singularities}
		Let $I \subset \mathbb{R}$ be an interval. A continuous family of piecewise smooth curves $\lbrace \gamma_t \rbrace_{t \in I}$ in $\mathbb{R}^3$  is said to be a solution of the flow \eqref{eq:Modified_Curve_Shortening} 
		if the following hold.
		\begin{itemize}
			\item[(i)] For all $t \in I$, any singular points of $\gamma_t$ 
			are singularities of $\phi$.
			\item[(ii)] Away from a finite set of times in $I$, each smooth component $\gamma^{(i)}_t$ of $\gamma_t$ satisfies \eqref{eq:Modified_Curve_Shortening}, i.e.
			$$\partial_t \gamma^{(i)}_t = \phi^{-1} ( \gamma^{(i)}_t) ''.$$
		\end{itemize}
	\end{definition}
	
Note that a piecewise smooth, embedded, planar curve $\gamma$ has a grading $\theta$, which is a lift of the angle that its tangent vector makes with a fixed line, and that is only defined where $\gamma'$ is well-defined.  We can then clearly extend the definition of almost calibrated to such curves. 
	
	The main result of this subsection is the following. 
	
	\begin{proposition}\label{prop:Flow_With_Surgeries}
Let $\gamma_0$ be an almost calibrated planar curve in some $2$-plane $P\subset\mathbb{R}^3$ and let $S$ be the singularities of $\phi$ in $P$.  Suppose that the $\gamma_0$ is either an embedded arc with endpoints in $S$ or asymptotic to a pair of distinct lines so that, in both cases, the interior of $\gamma_0$ does not meet $S$. Then, there is a continuous family of piecewise smooth almost calibrated curves 
		$\lbrace \gamma_t \rbrace_{t \geq 0}\subset P$,
		which is a solution of the flow \eqref{eq:Modified_Curve_Shortening}. 
		This family of curves is real analytic in space-time except for a finite set of spatial points which lie in $S$. Furthermore, there is a time $T$ such that for $t \geq T$ the number of smooth components of $\gamma_t$ stays constant, so the flow has no further singular times.
	\end{proposition}

\begin{remark}
An example of a flow through singularities produced by Proposition \ref{prop:Flow_With_Surgeries}, with three finite singular times, is shown in Figure \ref{fig:flow.with.sings}.
\end{remark}	

	\begin{figure}[ht]
		\centering
		\begin{center}
			\begin{tikzpicture}
							\node (p1) at (0,0) [inner sep=2pt,circle,draw=red,fill=red] {};
				\node (q1) at (1,2) [inner sep=2pt,circle,draw=red,fill=red] {};
				\node (q2) at (2,2.5) [inner sep=2pt,circle,draw=red,fill=red] {};
					\node (q3) at (3,3) [inner sep=2pt,circle,draw=red,fill=red] {};
					\node (q4) at (4.5,1.5) [inner sep=2pt,circle,draw=red,fill=red] {};
				\node (p2) at (5,0) [inner sep=2pt,circle,draw=red,fill=red] {};
				\node at (-0.3,-0.1) {$p_1$};
				\node at (5.35,-0.1) {$p_2$};
				\node at (1.1,1.7) {$q_1$};
				\node at (2,2.2) {$q_2$};
				\node at (2.9,2.7) {$q_3$};
				\node at (4.2,1.4) {$q_4$};
				\draw[black,thick,->] (p1) to node [swap] {} (p2);
				\draw[cyan,thick,->] (q1) to node {} (p1);
				\draw[cyan,thick,->] (q2) to node {} (q1);
				\draw[cyan,thick,->] (q3) to node {} (q2);
				\draw[cyan,thick,->] (q3) to node {} (q2);
				\draw[cyan,thick,->] (q4) to node {} (q3);
				\draw[cyan,thick,->] (q4) to node {} (q3);
				\draw[cyan,thick,->] (p2) to node {} (q4);
				\node at (2.5,-0.2) {$\ell$};
				\node at (0.5,0.5) {$\ell_1$};
				\node at (1.5,2) {$\ell_2$};
				\node at (2.5,2.5) {$\ell_3$};
				\node at (3.4,2.3) {$\ell_4$};
				\node at (4.6,0.5) {$\ell_5$};
				\draw[blue,thick,->] (p2) .. controls (4.5,3.5) and (4,4)
				  ..  (2.5,4) ; 
				  \draw[blue,thick] (2.5,4) .. controls (1.5,4) and (0.2,2)
				  ..  (p1) ; 
				  \node at (2.5,4.3) {$\gamma_0$};
\end{tikzpicture} $\qquad\qquad\qquad$
			\begin{tikzpicture}
							\node (p1) at (0,0) [inner sep=2pt,circle,draw=red,fill=red] {};
				\node (q1) at (1,2) [inner sep=2pt,circle,draw=red,fill=red] {};
				\node (q2) at (2,2.5) [inner sep=2pt,circle,draw=red,fill=red] {};
					\node (q3) at (3,3) [inner sep=2pt,circle,draw=red,fill=red] {};
					\node (q4) at (4.5,1.5) [inner sep=2pt,circle,draw=red,fill=red] {};
				\node (p2) at (5,0) [inner sep=2pt,circle,draw=red,fill=red] {};
				\node at (-0.3,-0.1) {$p_1$};
				\node at (5.35,-0.1) {$p_2$};
				\node at (1.1,1.7) {$q_1$};
				\node at (2,2.2) {$q_2$};
				\node at (2.9,2.7) {$q_3$};
				\node at (4.2,1.4) {$q_4$};
				\draw[black,thick,->] (p1) to node [swap] {} (p2);
				\draw[cyan,thick,->] (q1) to node {} (p1);
				\draw[cyan,thick,->] (q2) to node {} (q1);
				\draw[cyan,thick,->] (q3) to node {} (q2);
				\draw[cyan,thick,->] (q3) to node {} (q2);
				\draw[cyan,thick,->] (q4) to node {} (q3);
				\draw[cyan,thick,->] (q4) to node {} (q3);
				\draw[cyan,thick,->] (p2) to node {} (q4);
				\node at (2.5,-0.2) {$\ell$};
				\node at (0.5,0.5) {$\ell_1$};
				\node at (1.5,2) {$\ell_2$};
				\node at (2.5,2.5) {$\ell_3$};
				\node at (3.4,2.3) {$\ell_4$};
				\node at (4.6,0.5) {$\ell_5$};
				\draw[blue,thick] (p2) .. controls (4.8,1)  
				  ..  (q4) ;
				\draw[blue,thick,->] (q4) .. controls (4,2.5) and (3.1,3.6)
				  ..  (2.5,3.5) ;  
				  \draw[blue,thick] (2.5,3.5) .. controls (2,3.5) and (0.1,2.5)
				  ..  (p1) ;  
				  \node at (2.5,3.8) {$\gamma_{T_1}$};
\end{tikzpicture}

\vspace{20pt}

\begin{tikzpicture}
							\node (p1) at (0,0) [inner sep=2pt,circle,draw=red,fill=red] {};
				\node (q1) at (1,2) [inner sep=2pt,circle,draw=red,fill=red] {};
				\node (q2) at (2,2.5) [inner sep=2pt,circle,draw=red,fill=red] {};
					\node (q3) at (3,3) [inner sep=2pt,circle,draw=red,fill=red] {};
					\node (q4) at (4.5,1.5) [inner sep=2pt,circle,draw=red,fill=red] {};
				\node (p2) at (5,0) [inner sep=2pt,circle,draw=red,fill=red] {};
				\node at (-0.3,-0.1) {$p_1$};
				\node at (5.35,-0.1) {$p_2$};
				\node at (1.1,1.7) {$q_1$};
				\node at (2,2.2) {$q_2$};
				\node at (2.9,2.7) {$q_3$};
				\node at (4.2,1.4) {$q_4$};
				\draw[black,thick,->] (p1) to node [swap] {} (p2);
				\draw[cyan,thick,->] (q1) to node {} (p1);
				\draw[cyan,thick,->] (q2) to node {} (q1);
				\draw[cyan,thick,->] (q3) to node {} (q2);
				\draw[cyan,thick,->] (q3) to node {} (q2);
				\draw[cyan,thick,->] (q4) to node {} (q3);
				\draw[cyan,thick,->] (q4) to node {} (q3);
				\draw[cyan,thick,->] (p2) to node {} (q4);
				\node at (2.5,-0.2) {$\ell$};
				\node at (0.5,0.5) {$\ell_1$};
				\node at (1.5,2) {$\ell_2$};
				\node at (2.5,2.5) {$\ell_3$};
				\node at (3.4,2.3) {$\ell_4$};
				\node at (4.6,0.5) {$\ell_5$};
				\draw[blue,thick] (p2) .. controls (4.78,1)  
				  ..  (q4) ;
				  \draw[blue, thick,->] (q4) .. controls (4.3,2) and (3.3,3) .. (q3);
				  \draw[blue, thick,->] (q3) .. controls (2,3) and (0.1,2.5) .. (p1);
				  \node at (2.5,3.2) {$\gamma_{T_2}$};
\end{tikzpicture} $\qquad\qquad\qquad$ 
\begin{tikzpicture}
							\node (p1) at (0,0) [inner sep=2pt,circle,draw=red,fill=red] {};
				\node (q1) at (1,2) [inner sep=2pt,circle,draw=red,fill=red] {};
				\node (q2) at (2,2.5) [inner sep=2pt,circle,draw=red,fill=red] {};
					\node (q3) at (3,3) [inner sep=2pt,circle,draw=red,fill=red] {};
					\node (q4) at (4.5,1.5) [inner sep=2pt,circle,draw=red,fill=red] {};
				\node (p2) at (5,0) [inner sep=2pt,circle,draw=red,fill=red] {};
				\node at (-0.3,-0.1) {$p_1$};
				\node at (5.35,-0.1) {$p_2$};
				\node at (1.1,1.7) {$q_1$};
				\node at (2,2.2) {$q_2$};
				\node at (2.9,2.7) {$q_3$};
				\node at (4.2,1.4) {$q_4$};
				\draw[black,thick,->] (p1) to node [swap] {} (p2);
				\draw[cyan,thick,->] (q1) to node {} (p1);
				\draw[cyan,thick,->] (q2) to node {} (q1);
				\draw[cyan,thick,->] (q3) to node {} (q2);
				\draw[cyan,thick,->] (q3) to node {} (q2);
				\draw[cyan,thick,->] (q4) to node {} (q3);
				\draw[cyan,thick,->] (q4) to node {} (q3);
				\draw[cyan,thick,->] (p2) to node {} (q4);
				\node at (2.5,-0.2) {$\ell$};
				\node at (0.5,0.5) {$\ell_1$};
				\node at (1.5,2) {$\ell_2$};
				\node at (2.5,2.5) {$\ell_3$};
				\node at (3.4,2.3) {$\ell_4$};
				\node at (4.6,0.5) {$\ell_5$};
				\draw[blue,thick] (p2) .. controls (4.75,1)  
				  ..  (q4) ;
				    \draw[blue, thick,->] (q4) .. controls (4.2,2) and (3.2,3) .. (q3);
				  \draw[blue, thick] (q3) .. controls   (2.5,3) and (1.1,2.3) .. (q1);
				  \draw[blue, thick,->] (q1) .. controls   (0.9,2) and (0.1,1.5) .. (p1);
				  \node at (3,3.2) {$\gamma_{T_3}=\gamma_T$};
\end{tikzpicture}
\end{center}
\caption{Flow with  finite time singularities at $0<T_1<T_2<T_3=T$.}\label{fig:flow.with.sings}
\end{figure}
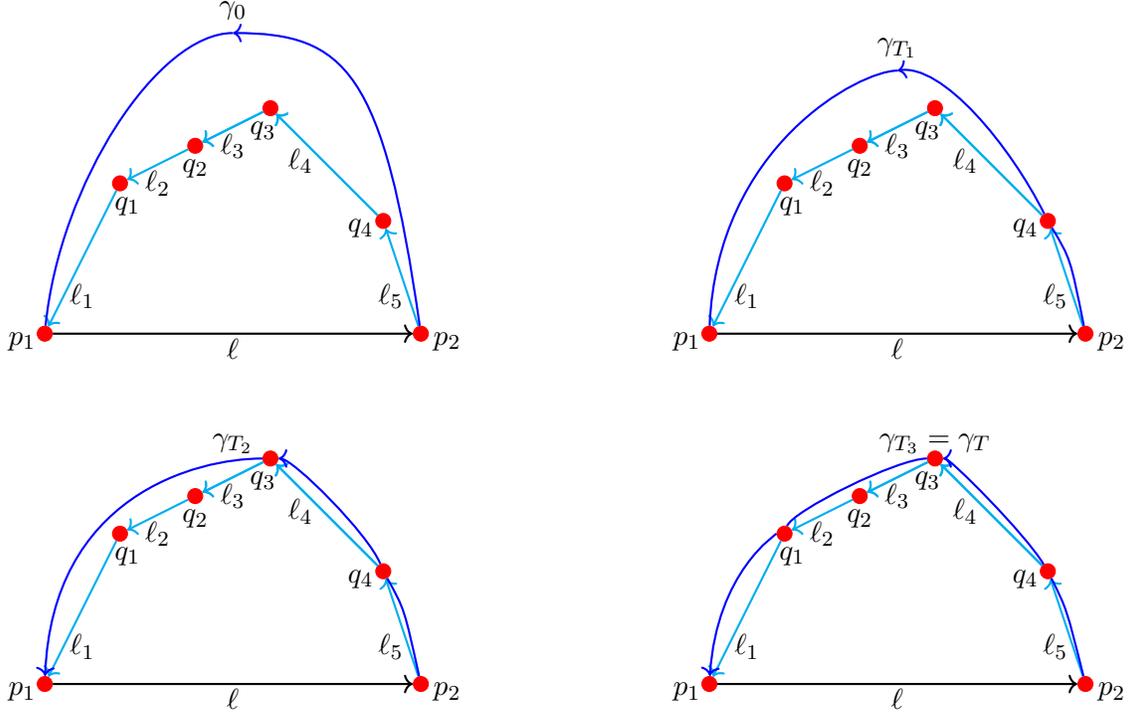

	\begin{proof} 
Let $\gamma_0$ be an almost calibrated arc in $P$ and write		$S=\lbrace p_1 , \ldots , p_k \rbrace$, where $p_1,p_2$ are the endpoints of $\gamma_0$. 
Given a family of curves $\lbrace \gamma_t \rbrace_{t \in [0,T_1)} \subset P$ starting at $\gamma_0$ and evolving through \eqref{eq:Modified_Curve_Shortening}, 
		Proposition \ref{prop:blowup_points}  shows
		that the flow exists as long as $\gamma_t$ does not meet $p_i$ for $i>2$. 
		
		Thus, since the flow \eqref{eq:Modified_Curve_Shortening} has short-time existence,  the first finite time singularity $T_1<\infty$ (if it exists) is characterized as the least $T_1>0$ such that
		$$\limsup_{t \nearrow T_1} \mathrm{dist}(\gamma_{t}, S \backslash \lbrace p_1 ,p_2 \rbrace) =0 . $$
		Let $S^{T_1}$ be the subset of $S$ contained in the limit set of $\gamma_t$ as $t \nearrow T_1$, i.e.
		$$S^{T_1}= S \cap \bigcap_{t \in [0, T_1)} \overline{\bigcup_{t' \in [t, T_1)} \gamma_{t'}} , $$
		which we write as $S^{T_1} = \lbrace p^{T_1}_1, \ldots , p^{T_1}_{k_1+1} \rbrace \subset S$.  Notice that $k_1+1>2=|\{p_1,p_2\}|$, since there is a singularity of the original flow $\gamma_t$ at $p_i$ for $i>2$.  Since $p_1 ,p_2  \in S^{T_1}$, with no loss of generality we may set $p_1^{T_1}=p_1$ and $p_{k_1+1}^{T_1}=p_2$ and note that $S\setminus S^{T_1}=k-k_1-1<k-2$.   Note also that, by Propositions \ref{prop:blowup_points}, the speed of the flow \eqref{eq:Modified_Curve_Shortening} for $\gamma_t$ tends to zero as $t\nearrow T_1$ at the points of $S^{T_1}$.
		
		We now describe the procedure for flowing past the singular time $T_1$.
		\begin{itemize}
			\item[(1)] We let $\gamma_{T_1}$ be the limit set of the curves $\gamma_t$ as $t \nearrow T_1$, which is a piecewise smooth curve with smooth components $\gamma_{T_1}^{(i)}$ with endpoints $p_{i}^{T_1},p_{i+1}^{T_1} \in S^{T_1}$ for $i=1,\ldots,k_1$, i.e.
			$$ \gamma_{T_1}:= \bigcap_{t \in [0, T_1)} \overline{\bigcup_{t' \in [t, T_1)} \gamma_{t'}} = \gamma^{(1)}_{T_1} \cup \gamma^{(2)}_{T_1} \cup \ldots \cup \gamma^{(k_1)}_{T_1}.$$ 
Note that $\gamma_{T_1}$ is $C^1$ near each point in $S^{T_1}$ by Theorem \ref{thm:Lawlor_Intro} and \cite{FSS.neck}*{Theorem 1.2}. 
\item[(2)] Note that each smooth arc $\gamma^{(i)}_{T_1}$ can be graded so that, away from the points of $S^{T_1}\setminus\{p_1,p_2\}$, a grading $\theta(T_1)$ of $\gamma_{T_1}$ is defined  so that $\theta(t)\to \theta(T_1)$ as $t\nearrow T_1$.  Given that $\gamma_0$ is almost calibrated and the grading $\theta(t)$ of $\gamma_t$ evolves through the heat equation $\gamma_t$ for $t<T_1$, there is a $\delta>0$ such that
$$\sup_{\gamma_t}\theta(t)-\inf_{\gamma_t}\theta(t)\leq\sup_{\gamma_0}\theta(0)-\inf_{\gamma_0}\theta(0)\leq \pi-\delta.$$
Therefore, by continuity, the grading $\theta(T_1)$ of $\gamma_{T_1}$ satisfies the same inequality above.  Hence, $\gamma_{T_1}$ is almost calibrated, as are all of the smooth arcs $\gamma^{(i)}_{T_1}$ for $i=1,\ldots,k_1$.
			\item[(3)] Let $i\in\{1,\ldots,k_1\}$. Since $\phi^{-1}(\gamma^{(i)}_{T_1})''$ tends to zero at the endpoints $p_i^{T_1},p_{i+1}^{T_1}$ of $\gamma^{(i)}_{T_1}$, we may restart the flow \eqref{eq:Modified_Curve_Shortening} at time $T_1$ with initial condition the almost calibrated arc $\gamma_{T_1}^{(i)}$, fixing the endpoints of the evolving arcs, and it will remain almost calibrated. According to Definition \ref{def:Flow singularities}, this means that the piecewise smooth almost calibrated curves $\gamma_{t}=\cup_{i=1}^k\gamma_t^{(i)}$, for $t>T_1$ solve \eqref{eq:Modified_Curve_Shortening} with initial condition $\gamma_{T_1}$.  Furthermore, as $\gamma_t$ solves a parabolic equation away from the points of $S^{T_1}$, it is real analytic except possibly at those points.
					\end{itemize}
		
We now proceed by induction, applying the procedure above to each independent flow $\gamma_{t}^{(i)}$ of smooth arcs with fixed endpoints lying in $S^{T_1}$ until the next finite time singularity $T_2>T_1$ (if it exists).
		Through this procedure, we obtain a continuous family $\{\gamma_t\}_{t\geq 0}$ of  piecewise smooth, almost calibrated curves solving \eqref{eq:Modified_Curve_Shortening} (in the sense of Definition \ref{def:Flow singularities}) with finite time singularities at times $T_1<T_2 < \ldots < T_l < \ldots$.  
 Note that the grading defining the almost calibrated condition flows through the singularities (cf.~\cite{FSS.neck}*{Theorem 1.2}) and each element of $S$ can occur at most once as a finite time singularity, otherwise $\gamma_t$ would contain a loop, which would violate the almost calibrated condition.  Therefore, using Proposition \ref{prop:blowup_points}, we see that the subset $S^{T_l}$ of $S$ contained in the limit set of $\gamma_t$ as $t\nearrow T_1$ satisfies $|S\setminus S^{T_l}|<k-l-1$, so there can be at most $k-2$ finite time singularities.
Hence, for sufficiently large $t$, the number of smooth components of $\gamma_t$ stays constant, which completes the proof in the case when $\gamma_0$ is an almost calibrated planar embedded arc.  

The case when $\gamma_0$ is instead an almost calibrated planar curve asymptotic to a pair of distinct lines follows from the same argument with minor modifications, using Proposition \ref{prop:blowup_points.2} in place of Proposition \ref{prop:blowup_points}.
	\end{proof}
		
	\begin{remark}
		The proof of Proposition \ref{prop:Flow_With_Surgeries} shows that the maximal number of smooth components of $\gamma_t$ for sufficiently large $t$ is $\# S -1$ because $\gamma_t$ cannot have loops.
	\end{remark}

	\subsection{Unstable curves} 
	We are now in a position to improve our results on strictly convex, strictly unstable, almost calibrated, compact curves in Theorem \ref{thm:Unstable} in two different directions. 
	First, we shall investigate the long-time behaviour for the flow of such unstable curves, and then use this to prove that all almost calibrated unstable curves develop finite time singularities.
	
	\begin{corollary}
	\label{cor:Unstable_1}
Let $P$ be a $2$-plane in $\mathbb{R}^3$ and let $S$ be the singularities of $\phi$ in $P$.
Let $\gamma_0$ be an almost calibrated, strictly convex arc in $P$ with endpoints $p_1,p_2\in S$, which otherwise does not meet $S$.  Let $\ell$ be the straight line connecting $p_1,p_2$ and let $\mathcal{R}$ be the open region in $P$ bounded by $\gamma\cup\ell$.  Then, there is a ordered set of straight lines $\{\ell_1, \ldots , \ell_k\}$ with consecutive endpoints such that $\gamma_t$ converges uniformly to their union,
		$$\lim_{t \to + \infty} \gamma_t = \ell_1 \cup \ldots \cup \ell_k,  $$
and $\ell	\cup \ell_1 \cup \ldots \cup \ell_k$ is the boundary of the convex hull of $S\cap \mathcal{R}$ in $P$. 
	\end{corollary}
	\begin{proof} After possibly applying a rotation and translation of $\mathbb{R}^3$ (which each induce isometries of the metric on the ambient hyperk\"ahler 4-manifold $X$), we may assume that $\ell$ lies along the $x_1$-axis in $\mathbb{R}^3$ and that $P$ is the plane with $x_3=0$. Without loss of generality, we may assume that $p_1$ has $x_1$ coordinate less than $p_2$ and that $\gamma$ lies in the region with $x_2\geq 0$.  
	
Let $\{\gamma_t\}_{t\geq 0}$ be the piecewise smooth almost calibrated curves solving \eqref{eq:Modified_Curve_Shortening}  given by Proposition \ref{prop:Flow_With_Surgeries} and let $T\geq 0$ be so that the number of smooth components of $\gamma_t$ stays constant for $t\geq T$ and there are no further singular times. 
By Proposition \ref{prop:Flow_With_Surgeries} and its proof, we may decompose $\gamma_t=\gamma_{t}^{(1)}\cup\ldots\cup\gamma_{t}^{(l)}$ for $t\geq T$ into smooth components and let $q_1,\ldots,q_{l-1}\in S\setminus\{p_1,p_2\}$ be ordered so that their $x_1$-coordinates are increasing and so that the curves $\gamma_{t}^{(i)}$ connect $p_1$ to $q_1$, $q_1$ to $q_2$ and so on until the last curve connects $q_{l-1}$ to $p_2$.   
Since strict convexity is preserved along the flow for finite time by Proposition \ref{prop:convexity}, we find that each of the curves $\gamma_{t}^{(i)}$ is strictly convex. We then set $\ell_i$ to be the straight line in $P$ with the same endpoints as $\gamma_{t}^{(i)}$ for $i=1,\ldots,l$.
		
Since the flows $\gamma_{t}^{(i)}$ are strictly convex and almost calibrated, but have no finite time singularities for $t\geq T$, we deduce from Theorem \ref{thm:Unstable} that $\gamma_T^{(i)}$ must be semi-stable.
By \cite{Lotay2020}*{Corollary 6.11},  a flow stable curve will converge along the flow \eqref{eq:Modified_Curve_Shortening} to the straight line connecting its endpoints.  Therefore, if $\gamma_{T}^{(i)}$ is flow stable, the flow $\gamma_{t}^{(i)}$ converges smoothly to $\ell_i$. 
		Hence, if all $\gamma_{T}^{(i)}$ are flow stable the proof is complete with $k=l$.

Suppose that there is some $i\in\{1,\ldots,l\}$ such that $\gamma_{T}^{(i)}$ is semi-stable but not flow stable.  The semi-stability of $\gamma_{T}^{(i)}$ means there must be no singularities of $\phi$ in the region bounded by $\gamma_{T}^{(i)}$ and $\ell_i$, but that there must be a singularity of $\phi$ lying on $\ell_i$.  
We may then decompose $\ell_i$ into a union of straight lines $\ell_i^{(1)}\cup\ldots\cup \ell_i^{(k_i)}$ with consecutive endpoints connecting singularities of $\phi$ so that there are no singularities of $\phi$ in the interiors of each $\ell_i^{(j)}$.  Notice that each $\ell_i^{(j)}$ must have the same angle as $\ell_i$.        
Therefore, since $\gamma_{t}^{(i)}$ exists smoothly for all $t\geq T$,  
the flow can only become singular at singularities of $\phi$,  and the Lagrangian angle evolves through the heat equation \eqref{eq:beta.evol}, we see that the Lagrangian angle must converge uniformly to a constant as $t\to\infty$ (cf.~\cite{ThomasYau}*{p.~1109--1110}), which is the same angle as $\ell_i$.  Hence, $\gamma_t^{(i)}$ converges uniformly to the union of straight lines $\ell_i=\ell_i^{(1)}\cup\ldots\cup \ell_i^{(k_i)}$ as $t\to\infty$, which completes the proof (with $k=\sum_{i=1}^lk_i$, setting $k_i=1$ if $\gamma_T^{(i)}$ is flow stable).
	\end{proof}
	
\begin{remark}
Figure \ref{fig:flow.with.sings} gives an example of the result of Corollary \ref{cor:Unstable_1}, showing how an initial curve converges to a union of five straight line segments $\cup_{j=1}^5\ell_j$.  Notice that $\ell_2,\ell_3$ in Figure \ref{fig:flow.with.sings} have the same angle, which gives an example of the flow semi-stable but not flow stable situation considered at the end of the proof of Corollary \ref{cor:Unstable_1}. 
\end{remark}
	
	\begin{corollary}
	\label{cor:Unstable_2}
		Let $\gamma_0$ be an almost calibrated, strictly unstable, planar arc in $\mathbb{R}^3$ so that its intersection with the singularities of $\phi$ consists of its endpoints. Then, the flow $\lbrace \gamma_t \rbrace_{t \in [0,+\infty)}$ obtained through Proposition \ref{prop:Flow_With_Surgeries} attains a finite time singularity.	
	\end{corollary}
	\begin{proof}  Let $S$ be singularities of $\phi$ in the 2-plane $P$ containing $\gamma_0$. 
		If $\gamma_0$ is strictly unstable, then there is a non-empty subset $\lbrace q_1 , \ldots , q_l \rbrace\subseteq S$ contained in the interior of the region  bounded by $\gamma_0$ and the straight line $\ell$ connecting its endpoints (notice that even if $\gamma_0$ and $\ell$ intersect, they   enclose a possibly disconnected region). Then, we choose an almost calibrated,  strictly convex, planar arc $\hat{\gamma}_0$ in $P$ with the same endpoints as $\ell$ (and $\gamma_0$), meeting no other elements of $S$ and bounding a segment of $\gamma_0$ which, together with $\ell$,  encloses some of the singularities $\lbrace q_1 , \ldots , q_{l'} \rbrace  \subseteq \lbrace q_1 , \ldots , q_l \rbrace$ of $\phi$.  (Such a curve $\hat{\gamma}_0$   exists: see Figure \ref{fig:barrier} for an example.)
		
		
		\begin{figure}[h]
				\centering
		\begin{center}
			\begin{tikzpicture}[scale=0.8]
							\node (p1) at (0,0) [inner sep=2pt,circle,draw=red,fill=red] {};
				\node (q1) at (1,1) [inner sep=2pt,circle,draw=red,fill=red] {};

				\node (q2) at (3,2) [inner sep=2pt,circle,draw=red,fill=red] {};
					\node (q3) at (6,-1) [inner sep=2pt,circle,draw=red,fill=red] {};
					\node (q4) at (8,1) [inner sep=2pt,circle,draw=red,fill=red] {};
				\node (p2) at (10,0) [inner sep=2pt,circle,draw=red,fill=red] {};
				\node at (-0.4,-0.1) {$p_1$};
				\node at (10.4,-0.1) {$p_2$};
				\node at (1,0.6) {$q_1$};
				\node at (3,1.6) {$q_2$};
				\node at (6,-0.6) {$q_3$};
				\node at (8,0.6) {$q_4$};
				\draw[black,thick,->] (p1) to node [swap] {} (p2);

				\draw[blue,thick,->] (5,0) .. controls (4.5,2.5) and (4,3)
				  ..  (2.5,3) ; 
				  \draw[blue,thick] (2.5,3) .. controls (1.5,3) and (0.2,0.5)
				  ..  (p1) ; 
				  \node at (2.9,3.3) {$\gamma_0$};
				  \draw[blue,thick] (p2) .. controls (9,2) and (8,3) .. (7,0);
				  \draw[blue,thick] (p2) .. controls (9,2) and (8,3) .. (7,0);
				  \draw[blue,thick] (7,0) .. controls (6.5,-2) and (5.5,-2) .. (5,0);
				  \draw[black,thick,->] (p2) .. controls (9.5,3.5) and (6,4)
				  ..  (5,4) ; 
				  \draw[black,thick] (5,4) .. controls (2 ,4)  and (0.5,3)				  ..  (p1) ; 
				  \node at (4.8,4.5) {$\hat{\gamma}_0$};
				  \node at (4.5,-0.3) {$\ell$};
\end{tikzpicture}
\end{center}

\vspace{-20pt}

	\caption{A strictly convex barrier curve $\hat{\gamma}_0$ for a strictly unstable curve $\gamma_0$.}	\label{fig:barrier}
		\end{figure}
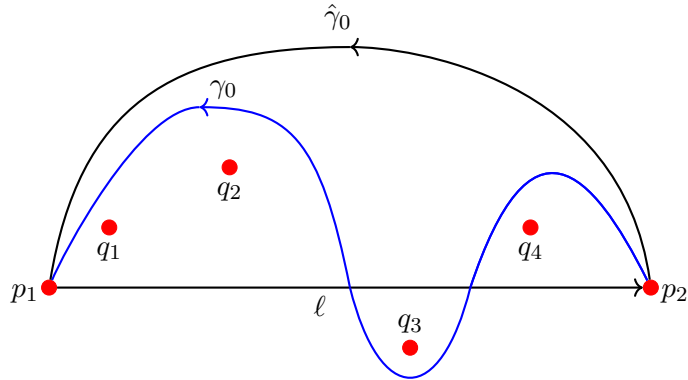
		
Let $\{\hat{\gamma}_t\}_{t\geq 0}$ and $\{\gamma_t\}_{t\geq 0}$ be the piecewise smooth solutions of the flow given by Proposition \ref{prop:Flow_With_Surgeries} starting at $\hat{\gamma}_0$ and $\gamma_0$ respectively. The maximum principle implies that $\hat{\gamma}_t$ and $\gamma_t$ do not intersect away from $S$, hence $\hat{\gamma}_t$ can be taken as a barrier for $\gamma_t$.
		
		Let $\Delta$ be the convex hull of the elements of $S$ contained in the closure of the region bounded by $\hat{\gamma} \cup \ell$. Its boundary $\partial \Delta$ is given by a union of straight lines $\ell \cup \ell_1 \cup \ldots \cup \ell_k$ and $\lim_{t \to +\infty} \hat{\gamma}_t = \ell_1 \cup \ldots \cup \ell_k$ by Corollary \ref{cor:Unstable_1}. 		Suppose that $q_i\in\partial\Delta$ for some $i\in\{1,\ldots, l'\}$.  Then, $\hat{\gamma}_t$ has a finite time singularity at $q_i$ by Corollary \ref{cor:Unstable_1} and its proof. The maximum principle then also implies that $ \gamma_t$ develops a finite time singularity at $q_i$ as required.   
		
Suppose instead that $\lbrace q_1 , \ldots , q_{l'} \rbrace$ is contained in the interior of $\Delta$. As $\lim_{t \to +\infty} \hat{\gamma}_t = \ell_1 \cup \ldots \cup \ell_k$ and $\hat{\gamma}_t$ is a barrier for $\gamma_t$ we find that for sufficiently large $T>0$ the curve $\gamma_T$ must pass through the interior of $\Delta$. Hence, if $\gamma_t$ does not develop a finite time singularity before $t=T$, we can find another strictly convex, almost calibrated curve $\hat{\gamma}_0'$ with the same endpoints as $\gamma_t$, totally contained inside $\Delta$, and which bounds the segment of $\gamma_T$ itself bounding $\lbrace q_1 , \ldots , q_{l'} \rbrace$ together with $\ell$. We may then replace $\hat{\gamma}_0$ with $\hat{\gamma}_0'$ to produce a new flow $\hat{\gamma}_t'$ by Proposition \ref{prop:Flow_With_Surgeries} which serves as a new barrier for $\gamma_t$.  Since $S$ is finite, this  procedure can be applied iteratively until at least one element of $\lbrace q_1 , \ldots , q_{l'} \rbrace$ is contained in the boundary of the convex hull of the singularities inside such a novel barrier curve. We are then in the previous situation above and so $\gamma_t$ does develop a finite time singularity as claimed.  
	\end{proof}

	\subsection{Long time behaviour for LMCF through singularities}
	
In this final subsection we put together the results of the previous subsections to prove Theorem \ref{thm:Flow_Intro}.  We begin with a natural definition following Definition \ref{def:Flow singularities}.
	
	\begin{definition}\label{def:Lagrangian_Mean_Curvature_Flow_With_Surgery}
		Let $L_0=\pi^{-1}(\gamma_0)$ be an embedded, circle-invariant, almost calibrated Lagrangian in the hyperk\"ahler 4-manifold $X$. Let $\lbrace \gamma_t \rbrace_{t \geq 0}$ be a piecewise smooth solution of the flow equation \eqref{eq:Modified_Curve_Shortening} in the sense of Definition \ref{def:Flow singularities}. We say that the continuous family $\lbrace L_t = \pi^{-1}(\gamma_t) \rbrace_{t \geq 0}$ is a \emph{Lagrangian mean curvature flow through singularities} starting at $L_0$.
	\end{definition}
	
	We now prove one of our main results  which, together with Theorem \ref{thm:Lawlor_Intro}, will account for items (a)--(d) in Theorem \ref{thm:Flow_Intro}, except for the convergence of currents in (c) of that result. 
	
	\begin{theorem}\label{thm:Flow}
		Let  
		$L_0=\pi^{-1}(\gamma_0)$ be a compact, connected, embedded, circle-invariant, almost calibrated Lagrangian in the hyperk\"ahler $4$-manifold $X$. 
		
\begin{itemize}
\item[(a)]		A compact, connected, circle-invariant, almost calibrated Lagrangian mean curvature flow through singularities $\lbrace L_t \rbrace_{t \geq 0}$ exists for all time and has a finite number of finite time singularities.
		\item[(b)] There is an $A_k$ chain $\{L_1^{\infty},\ldots,L_k^{\infty}\}$, in the sense of Definition \ref{dfn:Akchain}, of circle-invariant, embedded, special Lagrangian  spheres such that $\lbrace L_t \rbrace_{t \geq 0}$ uniformly converges to $L_1^\infty \cup \ldots \cup L_k^\infty$.  Moreover, if the grading on $L_0$ is a perfect Morse function, then the phases of these special Lagrangians can be arranged to be non-increasing.
		\item[(c)]	The number $k$ of special Lagrangians in the $A_k$ chain is exactly one if $\gamma_0$ is flow stable and strictly greater than one if $\gamma_0$ is unstable.
\end{itemize}
	\end{theorem}
	\begin{proof}
		Let $\gamma_0$ be the planar curve in $\mathbb{R}^3$ such that $L_0=\pi^{-1}(\gamma_0)$ and let $\ell$ be the straight line connecting its endpoints. By Proposition \ref{prop:Flow_With_Surgeries}, the piecewise smooth flow $\lbrace \gamma_t \rbrace_{t\geq 0}$ exists for all time. Hence, so does the Lagrangian mean curvature flow through singularities  $\lbrace L_t \rbrace_{t \geq 0}$, which gives (a).   Moreover, there is a finite time $T\geq 0$ so that the flow $\gamma_t$ (and hence $L_t$) has no singularities and the number of smooth components of $\gamma_t$ stays constant.

We may then decompose $\gamma_t$ for $t\geq T$ into smooth components as $\gamma_t=\gamma_t^{(1)}\cup\ldots\cup\gamma_t^{(l)}$ with the components ordered so that they have consecutive endpoints.  For each $i\in\{1,\ldots,l\}$, $\gamma_t^{(i)}$ is an almost calibrated planar arc only meeting the singularities of $\phi$ at its endpoints and with no singularity along the flow.  Corollary \ref{cor:Unstable_2} implies that $\gamma_t^{(i)}$ is semi-stable for all $i$.   
		
		By the same arguments leading to the conclusion of Corollary \ref{cor:Unstable_1} (which one may notice do not use any convexity assumption on $\gamma_t$) we deduce that there is an ordered set of straight lines $\{\ell_1,\ldots\ell_k\}$ with consecutive endpoints which are singularities of $\phi$ so that  
		$$\gamma_t \to \ell_1 \cup \ldots \cup \ell_k$$
uniformly as $t\to\infty$. Setting $L_i^{\infty}=\pi^{-1}(\ell_i)$ for $i=1,\ldots, k$ gives the first part of (b).

If the Lagrangian angle on $L_0$ is a perfect Morse function, then $\gamma_0$ is strictly convex by \eqref{eq:curv.angle}, so		
Corollary \ref{cor:Unstable_1} applies and $\ell \cup \ell_1 \cup \ldots \cup \ell_k$ is the boundary of the convex hull of the singularities of $\phi$ contained in the region enclosed by $\gamma_0$ and $\ell$.  Hence the non-increasing property in (b) can indeed be arranged as claimed.
		
		The fact that $k=1$ if $\gamma_0$ is flow stable follows from the proof of the circle-invariant Thomas--Yau conjecture in \cite{Lotay2020}.  The fact that $k>1$ if $\gamma_0$ is unstable is a consequence of Corollary \ref{cor:Unstable_2} if $\gamma_0$ is strictly unstable, and otherwise follows from the argument at the end of the proof of Corollary \ref{cor:Unstable_1} (again noticing that the convexity is not used there) since we are then in the semi-stable but not flow stable setting.
	\end{proof}
	
	\begin{remark}
	Figure \ref{fig:flow.with.sings} gives an example where the grading is a perfect Morse function and the $A_k$ chain (where $k=5$ in the example, corresponding to the lines $\{\ell_1,\ldots,\ell_5\}$) is arranged so that the phases are non-increasing.  Notice that the phases corresponding to $\ell_2,\ell_3$ are equal, which shows that one cannot always ensure that the phases are decreasing. 
	\end{remark}

	To complete the proof of Theorem \ref{thm:Flow_Intro}, it remains to prove the statements about the continuity and convergence as currents, which are consequences of the results we have proven so far.
	
	\begin{proposition}
		In the setting of Theorem \ref{thm:Flow}, the family $\lbrace L_t \rbrace_{t \in [0, + \infty)}$ varies continuously as an integral Lagrangian current with the following current convergence as $t\to\infty$:
		$$\lim_{t \to + \infty} L_t = L_1^\infty + \ldots + L_k^\infty .$$
	\end{proposition}
	\begin{proof}
		By Theorem \ref{thm:Flow}, there are $0<T_1<\ldots<T_l$ such that, for $t \in [0,+\infty) \backslash \lbrace T_1 , \ldots , T_l \rbrace$, $L_t$ is a union of smooth Lagrangians solving Lagrangian mean curvature flow, so the claimed continuity of $t\mapsto L_t$ follows immediately for these times. We are therefore left with proving continuity at the singular times $T_i$. For this, for any compactly supported 2-form $\alpha$ on $X$, we show that
		\begin{equation}\label{eq:Current}
			t \mapsto \int_{L_t} \alpha,
		\end{equation}
		is continuous at $t=T_i$ for $i=1,\ldots , l$.	Recall by the proof of Proposition \ref{prop:Flow_With_Surgeries} that, for $\epsilon>0$ sufficiently small,  
		$$L_{T_i}:=\bigcap_{s \in [T_{i}-\epsilon, T_i)} \overline{ \bigcup_{s' \in [s,T_i)} L_{s'} }.$$ 
		Hence, using the fact that $L_t = \pi^{-1}(\gamma_t)$ and $L_{T_i-\epsilon} \cup - L_{T_i}$ is the boundary of $\pi^{-1}(\overline{ \cup_{s \in [T_i -\epsilon,T_i ) } \gamma_{s} } ) $ and Proposition \ref{prop:Flow_With_Surgeries}, we find
		$$\big| \int_{L_{T_i-\epsilon}} \alpha - \int_{L_{T_i}} \alpha  \big| = \big| \int_{ \partial \pi^{-1}(\overline{ \cup_{s \in [T_i -\epsilon,T_i ) } \gamma_{s} } ) } \alpha \big| =  \big| \int_{ \pi^{-1}(\overline{ \cup_{s \in [T_i -\epsilon,T_i ) } \gamma_{s} } ) } d\alpha \big| \lesssim \epsilon \|d \alpha \|_{L^{\infty}}, $$
		thus proving that \eqref{eq:Current} is continuous. 
		
		Finally, the statement that $\lim_{t \to + \infty} L_t = L_1^\infty + \ldots + L_k^\infty$ as currents follows from the uniform convergence of $L_t$ to $L_1^\infty \cup \ldots \cup L_k^\infty$ from Theorem \ref{thm:Flow}.
	\end{proof}

	\section{Monotonicity of the Lagrangian angles}\label{s:monotonicity}
	
	Conjecture 3.34(e) in \cite{JoyceConjectures} states that the decomposition into special Lagrangians $L_1^\infty, \ldots , L_k^\infty$, as in Theorem \ref{thm:Flow_Intro}, should be chosen to have their phases satisfying $\theta_1 \geq \ldots \geq \theta_k$.  In Theorem \ref{thm:Flow_Intro}(d), when the grading on the initial Lagrangian is a perfect Morse function, we see that this ordering by phase coincides with the ordering as an $A_k$-chain. In this section we shall consider one example that shows how the decomposition of an initial Lagrangian using our flow through singularities works using the monotonicity of the Lagrangian angles, but which does not coincide with the ordering from the $A_k$-chain.
	
	We consider $X$ to be an ALE or ALF gravitational instanton obtained via the Gibbons--Hawking ansatz with $\phi$ having at least 4 planar singularities $\{p_1,p_2,q_1,q_2\}$ arranged as in Figure \ref{fig:bad.curve} and let $\gamma_0$ be the planar almost calibrated arc indicated there. 

			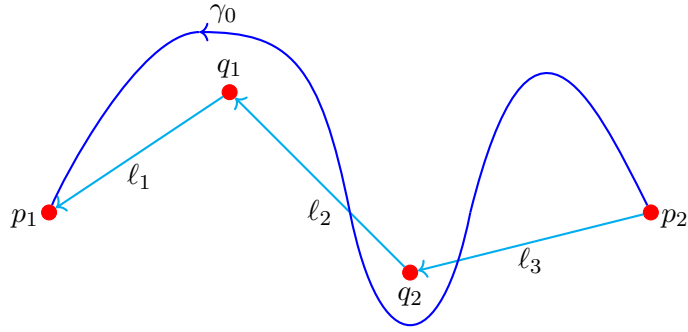
\begin{figure}[ht]
				\centering
		\begin{center}
			\begin{tikzpicture}[scale=0.8]
								\node (p1) at (0,0) [inner sep=2pt,circle,draw=red,fill=red] {};
							
				\node (q1) at (3,2) [inner sep=2pt,circle,draw=red,fill=red] {};
					\node (q2) at (6,-1) [inner sep=2pt,circle,draw=red,fill=red] {};
				\node (p2) at (10,0) [inner sep=2pt,circle,draw=red,fill=red] {};

				\node at (-0.4,-0.1) {$p_1$};
				\node at (10.4,-0.1) {$p_2$};
				\node at (3,2.4) {$q_1$};
				\node at (6,-1.4) {$q_2$};
							
				  \node at (2.9,3.3) {$\gamma_0$};
				  \draw[cyan,thick,->] (q1) to (p1);
				  \node at (1.5,0.6) {$\ell_1$};
				 \draw[cyan,thick,->] (q2) to  (q1);
				 \node at (4.5,0) {$\ell_2$};
				 \draw[cyan,thick,->] (p2) to (q2);
				 \node at (8,-0.8) {$\ell_3$};
				 \draw[blue,thick,->] (5,0) .. controls (4.5,2.5) and (4,3)
				  ..  (2.5,3) ; 
				  \draw[blue,thick] (2.5,3) .. controls (1.5,3) and (0.2,0.5)
				  ..  (p1) ; 
				  \draw[blue,thick] (p2) .. controls (9,2) and (8,4) .. (7,0);
				  \draw[blue,thick] (7,0) .. controls (6.5,-2.5) and (5.5,-2.5) .. (5,0);
						
\end{tikzpicture}
\end{center}

\vspace{-20pt}

	\caption{A flow through singularities starting at $\gamma_0$ converges to $\ell_1\cup\ell_2\cup\ell_3$.}	\label{fig:bad.curve}
		\end{figure}

By Theorem \ref{thm:Flow} there is a Lagrangian mean curvature flow through singularities $L_t=\pi^{-1}(\gamma_t)$ starting at $L_0=\pi^{-1}(\gamma_0)$ which as $t \to + \infty$ converges to an $A_k$ chain of special Lagrangian spheres.   We claim that this $A_k$ chain is $\cup_{j=1}^3L_j^{\infty}$ where $L_j^{\infty}=\pi^{-1}(\ell_j)$. To prove this, 
we define two triads: $(\gamma^1_-,\gamma^1_+,\hat{\gamma}^1)$ with vertices $(p_1,q_2,q_1)$ and $(\gamma^2_-,\gamma^2_+,\hat{\gamma}^2)$ with vertices $(p_2,q_1,q_2)$.  These triads are chosen so that $\hat{\gamma}^j$ is almost calibrated for $j=1,2$ and as shown in Figure \ref{fig:bad.curve.2}, which also shows the pacman disks associated with the triads.

		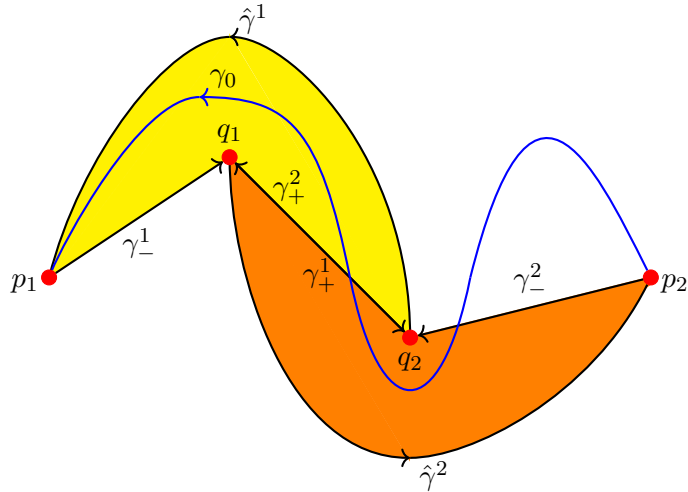
\begin{figure}[h]
				\centering
		\begin{center}
			\begin{tikzpicture}[scale=0.8]
			 
			\filldraw[yellow,draw=yellow,thick] (6,-1) -- (3.01,4) -- (3,2) -- (6,-1);
			\filldraw[yellow,draw=yellow] (0,0) -- (3,4) -- (3,2) -- (0,0);
			 \filldraw[fill=yellow, draw=black,thick,->] (6,-1) .. controls (6,2) and (4,4) .. (3,4);
			 \filldraw[fill=yellow, draw=black,thick] (3,4) .. controls (2,4) and (0.5,2) .. (0,0);
			 \filldraw[orange,draw=orange,thick] (3,2) -- (5.98,-3) -- (6,-1) -- (3,2);
			  \filldraw[orange,draw=orange,thick] (10,-0.1) -- (6,-3) -- (6,-1) -- (10,0);
			   \filldraw[fill=orange, draw=black,thick,->] (3,2) .. controls (3,0) and (4,-3) .. (6,-3);
			   \filldraw[fill=orange, draw=black,thick] (6,-3) .. controls (7,-3) and (9,-2) .. (p2);
							\node (p1) at (0,0) [inner sep=2pt,circle,draw=red,fill=red] {};

				\node (q1) at (3,2) [inner sep=2pt,circle,draw=red,fill=red] {};
					\node (q2) at (6,-1) [inner sep=2pt,circle,draw=red,fill=red] {};
				\node (p2) at (10,0) [inner sep=2pt,circle,draw=red,fill=red] {};

				\node at (-0.4,-0.1) {$p_1$};
				\node at (10.4,-0.1) {$p_2$};
				\node at (3,2.4) {$q_1$};
				\node at (6,-1.4) {$q_2$};
							
				  \node at (2.9,3.3) {$\gamma_0$};
				  \draw[black,thick,->] (p1) to node [below] {$\gamma^1_-$} (q1);
				 \draw[black,thick,->] (q1) to node [below] {$\gamma^1_+$}   (q2);
				 \draw[black,thick,->] (q2) to    (q1);
				 \draw[black,thick,<-] (q2) to node [above] {$\gamma^2_-$}   (p2);
				 \node at (4,1.5) {$\gamma_+^2$};
	
				 \draw[blue,thick,->] (5,0) .. controls (4.5,2.5) and (4,3)
				  ..  (2.5,3) ; 
				  \draw[blue,thick] (2.5,3) .. controls (1.5,3) and (0.2,0.5)
				  ..  (p1) ; 
				  \draw[blue,thick] (p2) .. controls (9,2) and (8,4) .. (7,0);
				  \draw[blue,thick] (7,0) .. controls (6.5,-2.5) and (5.5,-2.5) .. (5,0);
				
				\node at (3.4,4.3) {$\hat{\gamma}^1$};
				
				\node at (6.4,-3.3) {$\hat{\gamma}^2$};
\end{tikzpicture}
\end{center}


	\caption{Pacman disks for the flow through singularities starting at $\gamma_0$.}	\label{fig:bad.curve.2}
		\end{figure}

	The flows $\hat{\gamma}^j_t$ starting at $\hat{\gamma}^j$ have finite time singularities at $q_j$ by Proposition \ref{prop:Finite_Time_Singularity_Area}.  By construction, these flows act as barriers for $\gamma_t$, and since the only places where $\gamma_t$ can have finite time singularities are $q_1,q_2$, the flow $\gamma_t$ must have singularities at both points in finite time.

	Hence, there are two singular times $T_1 \leq T_2$ after which the family $\gamma_t$ consists of a union of three smooth flows of embedded arcs which exist for all time and converge to the straight lines connecting their endpoints.  
  We therefore have that $L_t$ converges to $\cup_{j=1}^3L_j^{\infty}$ where $L_j^{\infty}=\pi^{-1}(\ell_j)$ as claimed, and moreover
    $$\lim_{t \to + \infty} L_t = L_1^\infty + L_2^\infty + L_3^\infty.$$
    
Now, one needs to take care when exhibiting the limiting destabilizing configuration   for $L_0$ since, as in \cite{JoyceConjectures}*{Conjecture 3.34(e)}, one requires a monotonicity condition on the Lagrangian angles of $L_j^{\infty}$. 	
 Since the Lagrangian angles $\theta_j$ of $L_j^{\infty}$ are equal (up to the addition of a fixed constant, independent of $j$) to the angle $\ell_j$ makes with the line from $p_1$ to $p_2$, we see that $\theta_1>\theta_3>0>\theta_2$.  Hence, there is no natural ordering of the $L_j^{\infty}$ as a chain so that the desired monotonicity of the Lagrangian angles is achieved. 

 Instead, notice that we can decompose the Hamiltonian isotopy class of $L_0$ into a graded Lagrangian connect sum of the special Lagrangian spheres as $L_1^\infty \# (L_3^\infty \# L_2^\infty) $.   This shows that $L_1^\infty$ destabilizes the class of $L_0$ with ``quotient'' $L_3^\infty \# L_2^\infty$, which is itself destabilized by $L_3^\infty$ with ``quotient'' the stable special Lagrangian $L_2^\infty$. This sequence of destabilizations occurs in the correct order of monotonicity as $\theta_1>\theta_3>\theta_2$. The reader may wonder what would happen if we had considered a similar configuration to Figure \ref{fig:bad.curve} but with $\theta_3>\theta_1>\theta_2$ instead of $\theta_1>\theta_3>\theta_2$. This is explained by the fact that we   have Hamiltonian isotopies
 $$ L_1^\infty \# (L_3^\infty \# L_2^\infty) \sim L_0 \sim L_3^\infty \# (L_1^\infty \# L_2^\infty),$$
 and so we also have destabilizations ordered $L_3^\infty , L_1^\infty, L_2^\infty$, which is compatible with $\theta_3>\theta_1>\theta_2$.  Therefore,  though we have two different destabilizing configurations for $L_0$ using the special Lagrangians $L_j^{\infty}$, only one is  compatible with the required  monotonicity of the Lagrangian angles.
 
 In conclusion, the Lagrangian mean curvature flow starting at $L_0=\pi^{-1}(\gamma_0)$ as in Figure \ref{fig:bad.curve} shows how the limit of the conjectured Lagrangian mean curvature flow through singularities envisaged in \cite{JoyceConjectures} can be presented uniquely so that all of the conditions required there are satisfied, particularly involving the monotonicity of the Lagrangian angles.

	

\end{document}